\documentclass{amsart}
\usepackage{color, graphicx, enumerate, amssymb}
\usepackage{hyperref}
\usepackage{epsfig,wrapfig}
\usepackage{pxfonts}
\usepackage{graphicx}
\usepackage{eucal}
\usepackage[utf8]{inputenc}
\usepackage{ulem}

\usepackage{comment}

\usepackage{amsmath}
\usepackage{enumitem}

\newtheorem{theorem}{Theorem}[section] 

\newtheorem{lemma}[theorem]{Lemma}
\newtheorem{corollary}[theorem]{Corollary}
\newtheorem{definition}{Definition}[section]
\newtheorem{proposition}[theorem]{Proposition}
\newtheorem{remark}[theorem]{Remark}

\newcommand{\ee}{{\textbf{e}}}

\newcommand{\dist}{{\rm dist}}

\newcommand{\cK}{{\mathcal K}}
\newcommand{\bR}{{\mathbb R}}

\newcommand\norm[1]{\Arrowvert {#1}\Arrowvert}

\title[]{Global minimizers of the two-phase Bernoulli problem with the $p$-Laplace operator}

\author[M. Bayrami]{Masoud Bayrami}
\address{Department of Mathematical Sciences, Sharif University of Technology, Tehran, Iran}
\email{masoud.bayrami1990@sharif.edu}

\author[M. Fotouhi]{Morteza Fotouhi}
\email{fotouhi@sharif.edu}

\date{\today}

\begin{document}

\begin{abstract}
In this paper, we study the classification of Lipschitz global solutions for a two-phase $p$-Laplace Bernoulli problem. Specifically, we focus on the scenario where the \textit{interior} two-phase points of the global solution are non-empty.  Our results show that the expected $C^{1,\eta}$ regularity holds in a suitable neighborhood of certain two-phase points, which we refer to to as  \textit{regular} two-phase points. 
\end{abstract}

\keywords{Two-phase Bernoulli problem, singular/degenerate operator, viscosity solutions, Lipschitz global minimizers.}

\subjclass[2020]{35R35, 35B65, 35J60, 35J70.}

\thanks{M. Bayrami and M. Fotouhi was supported by Iran National Science Foundation (INSF) under project No. 4031333.}

\maketitle

\section{Introduction and main result}
We study the regularity of the free boundary arising from the minimization of the following two-phase functional
\begin{equation}
\label{00J_TP}
J_{\mathrm{TP}}(v,D) :=\int_{D} |\nabla v|^{p} + (p-1) \lambda_+^{p} \chi_{\{v>0\}} + (p-1) \lambda_-^{p} \chi_{\{v<0\}} \, dx,
\qquad v \in \cK.
\end{equation}
Here, $D$ is a bounded and smooth domain in $\mathbb{R}^n$, $\chi_A$ denotes the characteristic function of a set $A$, and the positive constants $\lambda_+, \lambda_->0$ are given, with $1<p<\infty$. 
The admissible class $\cK$ consists of all functions $v \in W^{1,p}(D)$ satisfying the boundary data $v=g$ on $\partial D$, for a prescribed boundary datum $g \in W^{1,p}(D)$.

Before proceeding, we recall some standard terminology and definitions:
\begin{itemize}
\item
A function $u : D \to \mathbb{R}$ is said to be a \textit{minimizer} of $J_{\mathrm{TP}}$ in $D$ if 
$$ J_{\mathrm{TP}}(u,D)\leq J_{\mathrm{TP}}(v,D), $$
for all $v \in \cK$. 

\item
The sets $\Omega^+_u = \{x \in D \, : \, u(x) > 0\}$ and $\Omega^-_u = \{x \in D \, : \, u(x) < 0\}$ are are referred to as the positivity and negativity sets of $u$, respectively. We also write $u^+ := \max\{u,0\}$ and $u^- := \max\{-u,0\}$.
\item
 The set $F(u):=\left(\partial \Omega^+_u \cup \partial \Omega^-_u \right) \cap D$ represents the free boundary of $u$.
 
\item
The set $\Gamma_{\mathrm{TP}}(u):= \partial \Omega^+_u \cap \partial \Omega^-_u \cap D$ is the set of two-phase points of the free boundary $F(u)$. For simplicity, we denote it by $\Gamma_{\mathrm{TP}}$.
\item
The boundary of positive and negative phases can be decomposed as 
$$ \partial \Omega^{\pm}_u \cap D = \Gamma^{\pm}_{\mathrm{OP}} \cup \Gamma_{\mathrm{TP}},$$
where $\Gamma^{+}_{\mathrm{OP}}:=\left(\partial \Omega^{+}_u \setminus \partial \Omega^{-}_u \right) \cap D$ and $\Gamma^{-}_{\mathrm{OP}}:=\left(\partial \Omega^{-}_u \setminus \partial \Omega^{+}_u \right) \cap D$ are the one-phase parts of the free boundary. 
\item
We say that $x_0 \in \Gamma_{\mathrm{TP}}$ is an \textit{interior} two-phase point, and denote it by $x_0 \in \Gamma^{\mathrm{int}}_{\mathrm{TP}}$, if
$$ | B_r(x_0) \cap \{u=0\} | = 0, \qquad \text{for some} \quad r>0. $$
\item
We say that $x_0 \in \Gamma_{\mathrm{TP}}$ is a \textit{branch} point, denoted by $x_0 \in \Gamma^{\mathrm{br}}_{\mathrm{TP}}$, if
$$ | B_r(x_0) \cap \{u=0\} | > 0, \qquad \text{for every} \quad r>0. $$
\end{itemize}

Any minimizer $u$ satisfies, in a certain weak sense, the following  problem
\begin{equation}
\label{OVERDETERMINED}
\begin{cases}
\Delta_{p} u := \mathrm{div}( |\nabla u|^{p-2} \nabla u ) = 0,  \quad & \text{in} \quad  \,\, \Omega_u^+ \cup \Omega_u^-, \\
|\nabla u^+|^{p}-|\nabla u^-|^{p}= \lambda_+^{p}-\lambda_-^{p}, \quad |\nabla u^+| \geq \lambda_{+}, \quad |\nabla u^-| \geq \lambda_{-}, \quad & \text{on} \quad \Gamma_{\mathrm{TP}}, \\
|\nabla u^+| = \lambda_{+},  \quad & \text{on} \quad \Gamma^{+}_{\mathrm{OP}}, \\
|\nabla u^-| = \lambda_{-}, \quad & \text{on} \quad \Gamma^{-}_{\mathrm{OP}},
\end{cases}
\end{equation}
where $\Delta_{p} u = \mathrm{div}( |\nabla u|^{p-2} \nabla u )$ is the $p$-Laplace operator; see \cite[Lemma 3.1]{fotouhi-bayrami2023}.
These types of problems are known as Bernoulli-type free boundary problems and arise in various models in fluid mechanics and heat conduction; see, for example,  \cite{MR682265, MR733897, MR740956, MR772122, karakhanyan2021regularity}.

For admissible functions in $\cK^+:=\{v \in \cK \, : \, v \geq 0\}$, the analogous one-phase functional and the associated  overdetermined problem, known as the one-phase Bernoulli problem, were first studied in \cite{MR618549} in the case $p=2$. 
Also, the case of uniformly elliptic quasilinear equations was later treated in \cite{MR752578}. 
Earlier results on the one-phase Bernoulli problem for the $p$-Laplace operator can be found in \cite{bayrami2024lipschitz, MR2133664,  MR2250499, fotouhi2024minimization, karakhanyan2021full, MR2399040}. 
We also refer to \cite{ferrari2022regularity} for the $p(x)$-Laplacian Bernoulli problem with a nontrivial right-hand side.
It is worth noting that, in general, for the one-phase Bernoulli problem, the $(n-1)$-Hausdorff dimension of the singular set of the free boundary is zero and that the free boundary is analytic in a suitable neighborhood of the regular points, \cite{MR2133664}.

The two-phase problem in the case $p=2$ was first studied in \cite{MR732100}, followed by a series of works by Luis Caffarelli, \cite{MR990856, MR1029856,  MR973745}; see also \cite{MR2145284}.
A comprehensive study and the review of earlier results in the case $p=2$ can be found in \cite{MR4285137, velichkov2019regularity}.

The main difficulty in dealing with \eqref{OVERDETERMINED} is that the governing operator $\Delta_p u$ is not uniformly elliptic. 
However, the regularity of the free boundary implies non-degeneracy of $|\nabla u|$ near the free boundary, which in turn restores uniform ellipticity for the $p$-Laplacian in a neighborhood of regular points. 
Without such regularity, it is difficult to establish non-degeneracy up to the free boundary.

For the minimizers of the following functional, instead of \eqref{00J_TP}
$$ J(u,D):=\int_{D} |\nabla u|^p + \lambda_+^p \chi_{\{u>0\}} + \lambda_-^p \chi_{\{u \leq 0 \}} \, dx=\int_{D} |\nabla u|^p + \left(\lambda_+^p-\lambda_-^p \right) \chi_{\{u>0\}} \, dx + \lambda_-^p \left| D \right|, $$
or, equivalently, after neglecting the constant term $\lambda_-^p \left| D \right|$, and setting $\Lambda:=\lambda_+^p-\lambda_-^p$, one obtains the functional
\begin{equation}
\label{AJ_TP}
u \mapsto \int_{D} |\nabla u|^p + \Lambda \chi_{\{u>0\}} \, dx,
\end{equation}
the best known results come from the stratification argument in \cite{MR3771123}. 
The important fact about the minimizers of the functional \eqref{AJ_TP} is that their free boundary does not contain any branch points. 
In \cite{MR2680176, MR2876248}, it is shown that Lipschitz free boundaries of minimizers of \eqref{AJ_TP} are $C^{1,\eta}$ for some $\eta \in (0,1)$. Moreover, \cite{MR3771123} establishes that the $(n-1)$-dimensional Hausdorff measure of the singular set of the free boundary is zero.

For minimizers of the functional \eqref{00J_TP}, where the presence of branch points is expected, a complete regularity theory is available in the case $p=2$; see \cite{MR4285137}. In that work, it is shown that in a suitable neighborhood of two-phase points (including both interior and branch points) $\partial\Omega_u^\pm$ are $C^{1,\eta}$ regular.

The generalization of this result, for any $p \in (1, \infty)$, is highly nontrivial due to the nonlinear and degenerate nature of the $p$-Laplace operator, as well as the lack of the monotonicity formulas when $p \neq 2$.

In \cite{fotouhi-bayrami2023}, we extend the main result of \cite{MR3771123} to functional \eqref{00J_TP}, which also covers the analysis of the behavior of the minimizers near the branch points. 
More precisely,  we prove the local Lipschitz regularity of the minimizers using a dichotomy argument. 
We first establish $C^{1,\eta}$ regularity of the flat part of the free boundary, and then show that  either the free boundary is flat or the solution exhibits linear growth at non-flat points.
In order to obtain $C^{1,\eta}$ regularity from flatness, we employ the linearization technique of De Silva \cite{MR2813524}. 
It is important to emphasize that in \cite{fotouhi-bayrami2023}, a linear improvement of flatness is used to obtain compactness of the linearizing sequences. In the present work, however, a stronger quadratic improvement of flatness is required.

The quadratic improvement of flatness deserves further attention, particularly in the study of free boundary problems where classical monotonicity formulas are unavailable. 
In a recent work, Savin and Yu \cite{savin2023regularity} established a stratification theorem for the singular set in the fully nonlinear obstacle problem. 
Their argument relies centrally on a self-improving mechanism: if a solution is sufficiently close to a paraboloid at a given scale, it becomes progressively closer at smaller scales.
Although the setting and objectives are different, this iterative improvement principle plays a similar role in the present work and underlies the quadratic improvement of flatness scheme developed here.

\subsection{Main results}

Our goal in this paper is to continue our previous research,  \cite{fotouhi-bayrami2023}, by classifying the Lipschitz global minimizers of \eqref{00J_TP} (which they are also Lipschitz solutions of \eqref{OVERDETERMINED} in $D=\mathbb{R}^n$). Our first main result is presented in the following classification theorem.

\begin{theorem}[Classification of Lipschitz global two-phase solutions]
\label{Ext02}
Let $ u $ be a Lipschitz global viscosity solution of \eqref{OVERDETERMINED}, i.e. with $D=\mathbb{R}^n$. 
Assume that 
$\Gamma_{\mathrm{TP}}^{\mathrm{int}}(u) \neq \emptyset$ (without loss of generality, assume that $0 \in \Gamma^{\mathrm{int}}_{\mathrm{TP}}(u)$), then $u$ is a two-plane solution of the form
\begin{equation*}
u(x)=\alpha \left( x \cdot \nu \right)^+-\beta \left( x \cdot \nu \right)^-,
\end{equation*}
for some $\nu \in \mathbb{S}^{n-1}$, with
\begin{equation*}
\alpha \geq \lambda_+, \qquad \beta \geq \lambda_-, \qquad \alpha^p-\beta^p=\lambda_+^p-\lambda_-^p.
\end{equation*}
\end{theorem}

The proof of Theorem \ref{Ext02} closely follows the approach of De Silva and Savin \cite{MR3998636}.  
As a result of this classification of Lipschitz global solutions to \eqref{OVERDETERMINED},  we obtain the regularity of the free boundary $F(u)=\left(\partial \Omega^+_u \cup \partial \Omega^-_u \right) \cap D$, for minimizers of $J_{\mathrm{TP}}$, around the  \textit{regular} two-phase points.

\begin{definition}\label{d}
We say that $x_0 \in \Gamma_{\mathrm{TP}}$ is a \textit{regular two-phase point}  if 
\begin{equation}\label{regular-point}
\liminf_{r\to0}\frac{| B_r(x_0) \cap \{u=0\} | }{|B_r(x_0)|} =0.
\end{equation}
\end{definition}

More precisely, we prove that, in a suitable neighborhood of the regular two-phase points, the sets $\Omega^+_u$ and $\Omega^-_u$ are two $C^{1,\eta}$ regular domains that meet along the closed set of the two-phase points $\Gamma_{\mathrm{TP}}$.

\begin{theorem}[Regularity of the regular part of the free boundary]
\label{T1}
Let $u:D \to \mathbb{R}$ be a minimizer of $J_{\mathrm{TP}}$ in $D$. Then, for every regular two-phase point $x_0$, there exists a radius $r_0>0$ (depending on $x_0$) such that $\partial \Omega^{\pm}_u \cap B_{r_0}(x_0)$ are $C^{1,\eta}$ graphs for any $\eta \in (0,\frac{1}{3})$. 
\end{theorem}

By definition, every interior two-phase point satisfies \eqref{regular-point} and is therefore a regular two-phase point. Consequently, Theorem \ref{T1} applies in a neighborhood of every interior two-phase point.
The situation at branch points is more subtle. At present, it is not known whether every branch point satisfies \eqref{regular-point}. 
Hence, branch points provide the only currently known obstruction to extending Theorem \ref{T1} to the entire two-phase free boundary in the nonlinear setting $p\neq 2$. Understanding the existence and structure of possible singular branch points remains an important open problem.

When $p=2$, the situation is significantly clearer.  Thanks to the ACF monotonicity formula and subsequent developments of the theory, all two-phase points are regular, a complete classification of global solutions is available, and the analogue of Theorem \ref{T1} holds throughout the entire two-phase free boundary; see \cite{MR4285137}.

\medskip

\subsection{Notation} 
We collect here the notation and conventions used throughout the paper.

\begin{description}[    style=nextline, 
  leftmargin=2.5cm,
  font=\normalfont]
\item[$\overline{U}$] the closure of a set $U$
\item[$\partial U$] the boundary of a set $U$
\item[$x'$] the vector $x'=(x_1, \dots, x_{n-1}) \in \mathbb{R}^{n-1}$, where $x=(x', x_n) \in \mathbb{R}^n$
\item[$B_r(x),\, B_r$] the open ball centered at $x$ with radius $r>0$, and $B_r := B_r(0)$
\item[$U_{a^{+},\nu}(x)$] the two-plane solution $U_{a^{+},\nu}(x):=a^+(x \cdot \nu)^+-a^-(x \cdot \nu)^-$, where $\left(a^+\right)^p-\left(a^-\right)^p=\lambda_+^p-\lambda_-^p$ and  $a^+ \ge \lambda_+$
\item[$\mathcal{U}_{\lambda_+,\lambda_-}$] the family of all two-plane solutions to \eqref{OVERDETERMINED}
\item[$\nu,\, \ee_n$] a unit vector $\nu \in \mathbb{S}^{n-1}$ and $\ee_n=(0, \dots, 0, 1) \in \mathbb{R}^n$
\item[$M=(M_{ij})$] a symmetric matrix $M=(M_{ij}) \in \mathcal{S}^{n \times n}$
\item[$\mathrm{tr}(M)$] the trace of $M$
\item[$P_{M,\nu}(x)$] the quadratic polynomial $P_{M,\nu}(x) := x \cdot \nu + \frac{1}{2} x^T M x$
\item[$V_{a^{+},M,\nu}(x)$] the two-phase quadratic
polynomial $V_{a^{+},M,\nu}(x):=a^+P^+_{M,\nu}(x)-a^-P^-_{M,\nu}(x)$, where $\left(a^+\right)^p-\left(a^-\right)^p=\lambda_+^p-\lambda_-^p$ and $a^+ \ge \lambda_+$
\item[$\partial_n$] the partial derivative in the $\ee_n$-direction
\item[$\partial_{nn}$] the second partial derivative in the $\ee_n$-direction
\item[$\partial^{\pm}_n$] the partial derivative in the $\ee_n$-direction taken from the side \hbox{$\{\pm x_n>0\}$}
\item[$\mathcal{L}_p (u)$] the linearized operator $\mathcal{L}_p (u) := \Delta u +(p-2)\partial_{nn} u$
\end{description}

\medskip

We briefly recall here the notion of blow-up. 
Let \( u \) be a minimizer of \( J_{\mathrm{TP}} \) in an open set \( D \). For \( x_0 \in F(u) \) and \( 0 < r < \operatorname{dist}(x_0, \partial D) \), we define the rescaled function
\[
u_{x_0,r}(x) := \frac{u(x_0 + rx)}{r},
\]
which is well‑defined for \( |x| < \frac{1}{r} \operatorname{dist}(x_0, \partial D) \) and vanishes at the origin. When \( x_0=0 \), we simply write \( u_r := u_{0,r} \).

For every \( R>0 \) and \( r \ll 1 \), the functions \( u_{x_0,r} \) are uniformly Lipschitz on \( B_R \), thanks to the local Lipschitz regularity of \( u \) (see \cite[Theorem 1.2]{fotouhi-bayrami2023}). 
Given \( r_k \to 0 \), the corresponding  family \( \{u_{x_0,r_k}\} \) is called a blow‑up sequence at \( x_0 \). 
By compactness,  after passing to a subsequence, there exist a Lipschitz function \( v : \mathbb{R}^n \to \mathbb{R} \) such that \( u_{x_0,r_k} \to v \) uniformly on every ball \( B_R \). Any such limit \( v \) is called  a blow‑up limit at \( x_0 \).

Throughout the paper, viscosity solutions are understood  in the sense of \cite[Lemma 3.1]{fotouhi-bayrami2023}. We also recall that viscosity and weak solutions of \( \Delta_p u = f \) are equivalent; see \cite[Theorem 2.2 and Remark 2.3]{MR4273843}.

\medskip

\subsection{Outline of the paper}
The paper is organized as follows: 
Section \ref{quadratic-improvement} is devoted to a Harnack-type inequality (Proposition \ref{Ext35}) and quadratic improvement of flatness (Proposition \ref{Ext09}) for solutions of \eqref{OVERDETERMINED}.
In Section \ref{Nonlinear dichotomy}, we introduce a nonlinear dichotomy which is crucial for the proof of the main results.
 Section \ref{Main-theorems} contains  the proofs of the main theorems, namely Theorem \ref{Ext02} and Theorem \ref{T1}.
Finally, in the Appendix, we establish a Harnack inequality for an auxiliary problem of $p$-Laplacian type, and present a Liouville type result for global solutions of the linearized operator, $\mathcal{L}_p$.


\section{Quadratic improvement of flatness}\label{quadratic-improvement}
In this section, we establish a quadratic improvement of flatness result. 
Before proceeding, we introduce some notation and definitions that will be used throughout the section.

We begin by defining a two-plane solution of  \eqref{OVERDETERMINED}
$$ U(x)=U_{a^{+},\nu}(x):=a^+(x \cdot \nu)^+-a^-(x \cdot \nu)^-, $$
for some $\nu \in \mathbb{S}^{n-1}$ and the constants $a^+, a^-$ that satisfy
$$ \left(a^+\right)^p-\left(a^-\right)^p=\lambda_+^p-\lambda_-^p, \qquad a^+ \geq \lambda_+, \quad a^-\geq \lambda_-. $$
We denote by $\mathcal{U}_{\lambda_+,\lambda_-}$ the family of all two-plane solutions.

Next, we introduce quadratic polynomials
$$ P_{M,\nu}(x) := x \cdot \nu + \frac{1}{2} x^T M x, $$
where $\nu \in \mathbb{S}^{n-1}$ and  $M=(M_{ij}) \in \mathcal{S}^{n \times n}$ is a symmetric matrix. In addition, they  satisfy the conditions
\begin{equation}
\label{CCFFNO1}
M\nu=0, \qquad  \mathrm{tr}(M)=0.
\end{equation}

We also define the associated two-phase quadratic polynomials
\begin{equation}
\label{QPNOT}
V_{a^{+},M,\nu}(x):=a^+P^+_{M,\nu}(x)-a^-P^-_{M,\nu}(x), \qquad \left(a^+\right)^p-\left(a^-\right)^p=\lambda_+^p-\lambda_-^p.
\end{equation} 
For convenience, we occasionally write $V_{a^{+},a^{-},M,\nu}$ when it is desirable to emphasize both slopes.
Notice that $ V_{a^{+},0,\nu}=U_{a^{+},\nu}$. 

It should be noted that the quadratic polynomials $P^+_{M,\nu}$ are not $p$-harmonic when $p\ne 2$, even though their linear counterpart $U_{a^{+},\nu}$ solves  \eqref{OVERDETERMINED}. 
This poses a challenge when attempting to extend the linear improvement of flatness to a quadratic one.

\begin{proposition}[Quadratic improvement of flatness]
\label{Ext09}
For every $1<p<\infty$, $\lambda_\pm>0$ and $0<L_0\le L_1$, there exists positive constants  $r_1$, $\epsilon_1$, and $\delta$  such that if 
 $u$ is a solution to \eqref{OVERDETERMINED} with the assumption $0 \in \Gamma_{\mathrm{TP}}(u)$ and furthermore it satisfies
\begin{equation}
\label{Ext10}
|u-V_{a^{+},M,\nu}| \leq \epsilon, \qquad \text{in} \quad B_1, 
\end{equation}
for some $0<\epsilon\leq \epsilon_1$, where $L_0\leq a^+\leq L_1$, and
\begin{equation}
\label{Ext11}
\|M\| \leq \delta \epsilon^{\frac{1}{2}}.
\end{equation}
Then
\begin{equation}
\label{Ext12}
|u-V_{\overline{a}^{+},\overline{M},\overline{\nu}}| \leq \epsilon r_1^{1+\alpha_1}, \qquad \text{in} \quad B_{r_1},
\end{equation}
for some $\alpha_1>0$, where
$$ |\overline{a}^{+}-a^{+}|, \, |\overline{\nu}-\nu|, \|\overline{M} - M\| \leq C \epsilon. $$
Here, the constants $r_1, \epsilon_1, \delta, \alpha_1$, and $C$ depend only on $n, p, \lambda_\pm, L_0$, and $L_1$. 
\end{proposition}

\begin{remark}
\label{LLREM1}
The conclusion of Proposition \ref{Ext09} can be iterated indefinitely. 
Indeed, if \eqref{Ext12} holds, then the rescaled function 
$$ u_r(x):=\frac{u(rx)}{r}, \qquad r=r_1, $$
satisfies
$$ |u_{r}-V_{\overline{a}^{+},r\overline{M},\overline{\nu}}| \leq \epsilon_r:=\epsilon r^{\alpha_1}, \qquad \text{in} \quad B_1, $$
provided that
$$ \|r \overline{M} \| \leq \delta \epsilon_r^{\frac{1}{2}}, \qquad \epsilon_r \leq \epsilon \leq \epsilon_1. $$
This condition is easily verified. Indeed, choosing $\epsilon_1$ sufficiently small, we have 
$$ \|r\overline{M}\| \leq \|r M \|+ C\epsilon r \leq r \delta \epsilon^{\frac{1}{2}}+C \epsilon r \leq \delta \epsilon_r^{\frac{1}{2}} = \delta (\epsilon r^{\alpha_1})^{\frac{1}{2}}. $$
\end{remark}
As a consequence of Proposition \ref{Ext09}, we obtain the following corollary.  
Arguing as in \cite[Section 5]{fotouhi-bayrami2023}, once a solution is sufficiently close to a two-plane solution, we can conclude that the interior two-phase set is a $C^{1,\eta}$ manifold.
\begin{corollary}
\label{Ext13}
There exist universal constants $\epsilon_1, \delta>0$ such that if \eqref{Ext10} and \eqref{Ext11} hold for some $0<\epsilon\leq \epsilon_1$, then $\Gamma_{\mathrm{TP}}(u)$ is a $C^{1,\eta}$ manifold with small $C^{1,\eta}$ norm in $B_{\frac{1}{2}}$. Moreover,
\begin{equation*}
|\nabla u^{+}(0)|\leq a^{+}+C \epsilon,
\end{equation*}
and
\begin{equation*}
\left| \frac{\nabla u^+(0)}{|\nabla u^+(0)|}-\nu \right| \leq C \epsilon,
\end{equation*}
where $C>0$ is a universal constant.
\end{corollary}


\subsection{Harnack-type inequality}

The proof of the quadratic improvement of flatness relies on the following Harnack-type inequality.

\begin{lemma} 
\label{Ext35}
For every $1<p<\infty$, $\lambda_\pm>0$, and $0<L_0\le L_1$, there exists positive constants  $\overline{\epsilon}, \delta, c>0$ such that the following holds. 
Let $u$ be a solution of \eqref{OVERDETERMINED} satisfying
\begin{equation*}
V(x+\sigma_1 \ee_n ) \leq u(x) \leq V(x+\sigma_2 \ee_n ), \qquad \text{in} \quad B_1,
\end{equation*}
where $0\le \epsilon=\sigma_2-\sigma_1\le\overline{\epsilon} $ and $V(x):=V_{a^+, M, \textbf{e}_n}(x)$ with $a^+\in [L_0, L_1]$.  
Assume further that $M$ satisfies
\begin{equation}
\label{Ext37}
\|M\| \leq \delta \epsilon^{\frac{1}{2}}, \qquad 0<\epsilon\leq \overline{\epsilon}.
\end{equation}
Then, one can find new constants $ \overline{\sigma}_1, \overline{\sigma}_2$ with
\[
0\le \overline{\sigma}_2 - \overline{\sigma}_1 \le c \epsilon
\]
such that 
\[
V(x+\overline{\sigma}_1 \ee_n ) \leq u(x) \leq V(x+\overline{\sigma}_2 \ee_n ), \qquad \text{in} \quad B_{\frac{1}{2}}.
\]
\end{lemma}

\begin{proof}
Without loss of generality, we may assume that  $\lambda_+\ge \lambda_-$.
This also implies that $a^+ \ge a^-$. 

Let $\overline{x}=\frac{1}{5}\textbf{e}_n$ and $\sigma=\frac{(\sigma_2-\sigma_1)a^+}2$. We will show that  
\begin{equation}
\label{Ext38}
u(\overline{x}) \geq V(\overline{x}+\sigma_1\ee_n) + \sigma \quad \Longrightarrow \quad 
u(x) \geq V(x+\overline{\sigma}_1\ee_n), \qquad \text{in} \quad B_{\frac{1}{2}},
\end{equation}
for $\overline{\sigma}_1 = \sigma_1 + c\epsilon$, where $0<c<1$ is a universal constant.
Similarly,
\begin{equation*}
u(\overline{x}) \leq V(\overline{x}+\sigma_2\ee_n)- \sigma  \quad \Longrightarrow \quad 
u(x) \leq V(x+\overline{\sigma}_2\ee_n), \qquad \text{in} \quad B_{\frac{1}{2}},
\end{equation*}
for $\overline{\sigma}_2 = \sigma_2 - c\epsilon$.

We prove only the first implication. The second follows by the same arguments.

\medskip

\textbf{Step 1:}
We first show that
\begin{equation} 
\label{MUL3}
u(x) \geq V(x+\sigma_1\ee_n) + c_0 \epsilon, \qquad \text{for all} \quad x \in B_{2r_0}(x_0),
\end{equation}
for some universal constants $c_0$ and  $r_0$, where $B_{4r_0}(x_0)\subset \{x_n>0\}$.  Moreover,
$x_0$ is either $\overline{x}$ or $x_0 = \overline{x}-2r_1\textbf{e}_n$ for some sufficiently small $r_1>0$.
To do this, we distinguish two cases:

\textbf{Case (i)}. Suppose $|\nabla u(\overline{x})| < \frac{1}{4}a^+$ for $\overline{x}=\frac{1}{5}\textbf{e}_n$. 
Then, there exists $r_1=r_1(n,p,L_0)>0$ such that $|\nabla u(x)| < \frac{1}{2}a^+$ in $B_{4r_1}(\overline{x})$; 
note that $u$ is universally bounded and $p$-harmonic in $B_{\frac{1}{10}}(\overline{x})$.
Now, by invoking Lemma \ref{auxiliaryLemma} in $B_{4r_1}(\overline{x})$ for $h:=u-V$, one can guarantee that
$$ u(x)-V(x+\sigma_1\ee_n) \geq C^{-1} \left( u(\overline{x})-V(\overline{x}+\sigma_1\ee_n) \right)-a^+r_1, \qquad \text{in} \quad B_{r_1}(\overline{x}), $$  
for an appropriate universal constant $C=C(n,p)>0$.
By the assumption \eqref{Ext38}, we obtain that for all $x \in B_{r_1}(\overline{x})$
$$ 
\begin{aligned}
C^{-1} \sigma - a^+r_1 & \leq u(x)-V(x+\sigma_1\ee_n) \\
& \leq u(x-2r_1\textbf{e}_n)-V(x+(\sigma_1-2r_1)\textbf{e}_n) \\
& \quad + \left( V(x+(\sigma_1-2r_1)\textbf{e}_n)-V(x+\sigma_1\ee_n) \right) + 2r_1 \norm{\nabla u}_{L^\infty(B_{4r_1}(\overline{x}))} \\
& \leq u(x-2r_1\textbf{e}_n)-V(x+(\sigma_1-2r_1)\textbf{e}_n)  -2a^+r_1 + 2r_1 \norm{\nabla u}_{L^\infty(B_{4r_1}(\overline{x}))} \\
& \leq u(x-2r_1\textbf{e}_n)-V(x+(\sigma_1-2r_1)\textbf{e}_n) - a^+r_1.
\end{aligned}
$$
Thus, with $x^*=\overline{x}-2r_1\textbf{e}_n$, we get
\begin{equation}
\label{MUL1}
C^{-1} \sigma \leq u(x)-V(x+\sigma_1\ee_n), \qquad \text{for all} \quad x \in B_{r_1}(x^*).
\end{equation}

\medskip

\textbf{Case (ii)}. Suppose $|\nabla u(\overline{x})| > \frac{1}{4}a^+$. 
By the interior gradient estimate, we know that $|\nabla u|$ is bounded in $B_{\frac{1}{10}}(\overline{x})$, and  there exists a constant $0<r_0=r_0(n,p,L_0,L_1)$, with $8r_0 \leq \frac{1}{10}$ such that
$$ \frac{1}{8}a^+ \leq |\nabla u(x)| \leq Ca^+, \qquad \text{for all} \quad x \in B_{8r_0}(\overline{x}), $$
for an appropriate universal constant $C=C(n,p)>0$. Now, $u$ will be the weak solution to the following uniformly elliptic equation
$$ \sum_{i,j=1}^n  \theta_{ij}  \partial_{x_ix_j} u = 0, \qquad \text{in} \quad B_{4r_0}(\overline{x}), $$
with $\theta_{ij}=\delta_{ij} + (p-2)|\nabla u|^{-2} \partial_{x_i}u \partial_{x_j}u $. 
Consequently, the function $h=u-V$ satisfies 
$$ 
\begin{aligned}
 \sum_{i,j=1}^n  \theta_{ij}  \partial_{x_ix_j} h =- a^+ \sum_{i,j=1}^n  \theta_{ij} M_{ij}  & =  -a^+ \mathrm{tr}(M)- a^+(p-2)\frac{\left(\nabla u \right)^TM\nabla u}{|\nabla u|^2} \\
& = -\frac{a^+(p-2)} {|\nabla u|^2} (\nabla u+Mx)^T M\nabla h - a^+(p-2) \frac{x^T M^3x}{|\nabla u|^2}.
\end{aligned}
$$
Then, recalling \eqref{Ext37}, and applying Harnack's inequality (see e.g. \cite[Chapter 9]{MR1814364}),  
one can guarantee that
\begin{equation}
\label{MUL2}
C^{-1} \sigma \leq u(x)-V(x+\sigma_1\ee_n), \qquad \text{for all} \quad x \in B_{r_0}(\overline{x}).
\end{equation}

\medskip

Thus, summarizing \eqref{MUL1} and \eqref{MUL2}, we get \eqref{MUL3} with $x_0=x^*$, or $x_0=\overline{x}$, and the  universal constant $c_0$.

\medskip

\textbf{Step 2:}
In this step, we will construct a barrier. 
For a given positive constant $\gamma>0$, let us define
$$ \psi(x)=c \left(r_0^{-\gamma}  -  |x-x_0|^{-\gamma} \right), $$
in the closure of the annulus $A:=B_{\frac{3}{4}}(x_0) \setminus B_{r_0}(x_0)$. 
Also, the positive constant $c$ is chosen so that  $\psi$ satisfies the following boundary conditions
\[\left\{\begin{array}{ll}
\psi= 1,  \quad & \text{on} \quad \partial B_{\frac{3}{4}}(x_0),\\
\psi= 0,  \quad & \text{on} \quad \partial B_{r_0}(x_0).
\end{array}\right.\] 
Extend $\psi$ to be equal to zero on $B_{r_0}(x_0)$. We can choose $\gamma=\gamma(n,p)$ such that 
\begin{equation}
\label{MUL1.5}
\Delta_p(P - \epsilon \psi)\ge c(n,p)\epsilon>0, \qquad \text{in} \quad A,
\end{equation}
for sufficiently small $\epsilon>0$, where $P(x)=\sigma_1+ x \cdot \ee_n+\frac12 x^TMx$.
To see this, we compute
\begin{align*}
\Delta_p(P - \epsilon \psi)  & =\mathrm{div}\left(|\ee_n +Mx - \epsilon\nabla \psi|^{p-2} (\ee_n +Mx - \epsilon\nabla \psi)\right)\\
& =|\ee_n +Mx - \epsilon\nabla \psi|^{p-4}\Big( - \epsilon |\ee_n +Mx - \epsilon\nabla \psi|^{2} \Delta \psi \\
& \quad + (p-2)(\ee_n +Mx - \epsilon\nabla \psi)^T(M - \epsilon D^2\psi)(\ee_n +Mx- \epsilon\nabla \psi)\Big) \\
& = -\epsilon \left(\Delta \psi + (p-2)\partial_{nn} \psi \right) + O(\epsilon^{\frac{3}{2}}).
\end{align*}
On the other hand,
\[
\Delta \psi + (p-2)\partial_{nn} \psi = -\frac{c\gamma}{|x-x_0|^{\gamma+4}} \left( (\gamma-n-p+4)|x-x_0|^2 + (p-2)(\gamma+2) |(x-x_0)\cdot\ee_n|^2\right).
\]
If we choose $\gamma > n+p-4$, when $p\ge 2$, and $\gamma > \frac{n-p}{p-1}$, when $1<p < 2$, then \eqref{MUL1.5} holds for small $\epsilon>0$.

\medskip

\textbf{Step 3:}
Now, define for $t \geq 0$
\[
v_t(x):= (1+\epsilon \ell) a^+\left(P(x)+\epsilon\tilde c(t- \psi(x)-k)\right)^+ 
-a^-\left(P(x)+\epsilon\tilde c(t- \psi(x)-k)\right)^-,
\]
where  $P(x)=\sigma_1+ x \cdot \ee_n+\frac12 x^TMx$, $k=\frac12(1-\|\psi\|_{L^{\infty}(B_{\frac{1}{2}}(0))})$, $\tilde c = \frac{c_0}{a^+}$ for $c_0$ in \eqref{MUL3},
and $\ell$ is sufficiently small so that   
\[
\ell P(x) \le \tilde c \min\left( \frac k2,  cr_0^{-\gamma}(1-2^{-\gamma})\right), \qquad \text{in} \quad B_{\frac{3}{4}}(x_0).
\]
In this step, we show that
$$ v_0(x)  \leq u(x), \qquad x \in \overline{B_{\frac{3}{4}}(x_0)}. $$
We first consider $x\in B_{2r_0}(x_0)$. By \eqref{MUL3},
\[
v_0(x) \le (1+\epsilon \ell) a^+P(x)^+ \le  V(x+\sigma_1\ee_n) + a^+ \tilde c\epsilon = V(x+\sigma_1\ee_n) + c_0 \epsilon \le u(x).
\]
Next, let $x\in B_{\frac{3}{4}}(x_0)\setminus B_{2r_0}(x_0)$. 
In this region,
$$\tilde c \psi(x) \ge \tilde c cr_0^{-\gamma}(1-2^{-\gamma}) \ge \ell P(x).$$ 
Now, assume that $P(x) - \epsilon \tilde c (\psi(x)+k)\ge 0$, hence
\[\begin{split}
v_0(x) & =(1+\epsilon \ell) a^+ \left(P(x) - \epsilon \tilde c( \psi(x)+k)\right) \\
& \leq V(x+\sigma_1\ee_n) + \epsilon a^+ \left(\ell  P(x) -(1+\epsilon \ell)\tilde c\psi(x)\right)\\
& \leq V(x+\sigma_1\ee_n) \le u(x).
\end{split}\]
If instead  $P(x)  - \epsilon \tilde c (\psi(x)+k)\le 0$, the result is obvious when $0\le P(x)$; because $v_0(x) \le 0\le a^+ P(x) =V(x+\sigma_1\ee_n)\le u(x)$. If $P(x)\le 0$,
\[
v_0(x) = a_-(P(x) - \epsilon \tilde c (\psi(x)+k)) \le a_-P(x) = V(x+\sigma_1\ee_n) \le u(x).
\]

\textbf{Step 4:}
Let $\overline{t}$ be the largest $t \geq 0$ such that
$$ v_t(x) \leq u(x), \qquad \text{in} \quad \overline{B_{\frac{3}{4}}(x_0)}. $$
We claim that $\overline{t} \geq 1$.
Suppose this is true, then we will have
\[\begin{split}
 u(x) \geq  v_{1}(x) & =(1+\epsilon \ell) a^+\left(P(x)+\epsilon\tilde c(1- \psi(x)-k)\right)^+ 
-a^-\left(P(x)+\epsilon\tilde c(1- \psi(x)-k)\right)^- \\
 & \ge V(x)+c \epsilon,
\qquad \text{in} \quad B_{\frac{1}{2}}(0) \Subset B_{\frac{3}{4}}(x_0), 
\end{split}\]
where $c$ is a universal constant, and therefore we obtain \eqref{Ext38}. 
Note that in the last inequality, we have used that $\|\psi\|_{L^{\infty}(B_{\frac{1}{2}}(0))} =1-2k$, and $a^+ \ge a^-$.

Suppose now that $\overline{t} < 1$. Then there exists $\tilde{x}\in \overline{B_{\frac{3}{4}}(x_0)}$ such that
\begin{equation}
\label{ContrTouch}
v_{\overline{t}}(\tilde{x})=u(\tilde{x}).
\end{equation}
We show that such a touching point must lie in $\overline{B_{r_0}(x_0)}$. Indeed, since $\psi =1$ on $\partial B_{\frac{3}{4}}(x_0)$, from the definition of $v_t$, we get that for $t < 1$,
\[
\begin{split}
v_t(x)  & =(1+\epsilon \ell) a^+\left(P(x)+\epsilon\tilde c(t- 1-k)\right)^+ 
-a^-\left(P(x)+\epsilon\tilde c(t- 1-k)\right)^- < V(x). 
\end{split}
\]
Also, $\tilde{x}$ cannot belong to the annulus $A$. Indeed, in $\{x \in A \, : \, u(x)\ne 0 \}$, we have $\Delta_p u(x)=0$ and by the definition of viscosity solution for $p$-harmonic function, it follows that $\Delta_p (v_{\overline{t}})(\tilde{x}) \le 0$ at touching point.
This contradicts \eqref{MUL1.5} for $\epsilon$ small enough. 

Finally, if $\tilde{x} \in \{u=0\}$, we check the free boundary condition at $\tilde x$.
\begin{align*}
| \nabla v_{\overline{t}}^+ (\tilde x) |^p - |\nabla v_{\overline{t}}^- (\tilde x)|^p  & = \left(a_+^p(1+\epsilon\ell)^p-a_-^p \right) \left|\ee_n + M\tilde{x} - \epsilon \tilde c \nabla \psi (\tilde x) \right|^p\\
& = \left(\lambda_+^p - \lambda_-^p + pa_+^p \epsilon \ell + O(\epsilon^2) \right) \left(1+\frac{p}{2}|M\tilde{x}|^2 - p \epsilon \tilde c  \partial_n \psi(\tilde{x}) )+O(\epsilon^{\frac32})\right). 
\end{align*}
One can verify that 
\[
\partial_n\psi(\tilde x)= c\gamma |\tilde x- x_0|^{-\gamma-2}(\tilde x-x_0)\cdot\ee_n  < - C<0,
\]
for a universal constant $C$; note that $\tilde x_n=O(\epsilon^{\frac{1}{2}})$. 
Taking into account $\|M\| \le \delta\epsilon^{\frac12}$ and $\lambda_+\ge \lambda_-$, we get
$$ |\nabla v_{\overline{t}}^+(\tilde x)|^p - |\nabla v_{\overline{t}}^-(\tilde x)|^p >\lambda_+^p-\lambda_-^p,$$
as long as $\delta, \epsilon$  are small enough.

Therefore, $\tilde{x} \in \overline{B_{r_0}(x_0)}$, and  for $\epsilon$ sufficiently small
\[
\begin{split}
 v_{\overline{t}}(\tilde{x}) & =(1+\epsilon \ell) a^+\left(P(\tilde x)+\epsilon\tilde c(\overline{t}-k)\right) \\
 & < V(\tilde{x}) + \epsilon a^+ \left(\ell P(\tilde x) + \tilde c(1+\epsilon\ell)(1-k)\right)\\
 & \le V(\tilde{x}) + \epsilon c_0  \left( \frac k2 + (1+\epsilon\ell)(1-k)\right) \\
 & < V(\tilde{x})+  \epsilon c_0 \le u(\tilde x),
\end{split}
\]
contradicting \eqref{ContrTouch}; where in the last inequality, we have used \eqref{MUL3}.
\end{proof}


\subsection{Constructing comparison functions}

In this subsection, our goal is to construct suitable strict subsolutions to the two-phase problem in $B_1$ that are $\epsilon$-perturbations of $V$. Similarly, strict supersolutions can be also constructed.

Given $\epsilon, \delta>0$, assume that
$$\nu=\textbf{e}_n, \qquad \|M\| \leq \delta \epsilon^{\frac{1}{2}}, $$
and due to the condition \eqref{CCFFNO1} on $\nu$ and $M$, we  have:
$$ M\textbf{e}_n=0, \qquad \mathrm{tr}(M)=0.$$
In the sequel, we  write $x=(x', x_n) \in \mathbb{R}^n$, where $x'=(x_1, \cdots, x_{n-1}) \in \mathbb{R}^{n-1}$.

Let $s,t,A \in \mathbb{R}$, $\xi' \in \mathbb{R}^{n-1}$ and $N=(N_{ij})$ be a symmetric  $n \times n$ matrix satisfying
\begin{equation}
\label{NMatrixCon}
\mathrm{tr}( N) + (p-2) N_{nn} >0, \qquad  \text{and} \qquad N\textbf{e}_n= N_{nn}\textbf{e}_n.
\end{equation}
Given $\epsilon>0$, define
$$ \overline{a}^+:=a^+(1+\epsilon s), \qquad \overline{a}^-=a^-(1+\epsilon t), \qquad \overline{M}:=M+2\epsilon N, $$
and 
$$ Q(x):= x \cdot \textbf{e}_n + \frac{1}{2} x^T \overline{M} x + A \epsilon + \epsilon \xi' \cdot x', $$
or equivalently,
$$ Q(x)= P_{\overline{M},\textbf{e}_n+ \epsilon (\xi',0)}(x) + A \epsilon. $$
We then define
\begin{equation}
\label{Ext32}
v(x):=\overline{a}^+ Q^+-\overline{a}^-Q^-.
\end{equation}

\begin{lemma}
\label{Ext33}
Suppose that the symmetric matrix $N$ satisfies \eqref{NMatrixCon}. If $s$ and $t$ satisfy
\begin{equation}
\label{Ext34}
\left(a^+ \right)^p s - \left(a^- \right)^{p}  t >  \delta^2 \left( \left(a^+ \right)^p+ \left(a^- \right)^p \right),
\end{equation}
then the function $v$ defined in \eqref{Ext32} is a strict subsolution to \eqref{OVERDETERMINED} in $B_1$ for all sufficiently small  $\epsilon$.
\end{lemma}

\begin{proof}
In $ \Omega_v^{+} \cap B_1$, we have 
$$ 
\begin{aligned}
\Delta_p  Q^{+}  & = \mathrm{div}\left( \left|\nabla \left( Q^{+}\right) \right|^{p-2} \nabla Q^{+}  \right) \\
& =  \left| \nabla Q^{+} \right|^{p-4} \left( \left| \nabla Q^{+}  \right|^{2}\Delta Q^{+} + (p-2) (\nabla Q^{+})^T D^2 Q^{+}  \nabla Q^{+} \right) \\
& =  \left| \nabla Q^{+}  \right|^{p-4} \left(  \left| \nabla Q^{+}  \right|^{2} \mathrm{tr}(M+2\epsilon N) +(p-2) \left( \nabla Q^{+} \right)^T (M+2 \epsilon N) \left( \nabla Q^{+} \right) \right) \\
& = \left| \nabla Q^{+}  \right|^{p-4} \Big[(p-2)x^TM^3x+2\epsilon \big (\left| Mx+ \textbf{e}_n \right|^{2} \mathrm{tr}( N) +(p-2)(Mx+ \textbf{e}_n )^TN(Mx+ \textbf{e}_n ) \\
&\ \ \  + (p-2)(2Nx+\xi')^TM^2x\big)
+ O(\epsilon^2) \Big] \\
& \geq  \left| \nabla Q^{+}  \right|^{p-4} \left[ -(p-2) \delta^3\epsilon^{\frac{3}{2}} + 2\epsilon \left(\mathrm{tr}( N) + (p-2)N_{nn}-2\delta \epsilon^{\frac{1}{2}}\|N\| \right) + O(\epsilon^2)\right] \\
&= 2\epsilon \left| \nabla Q^{+}  \right|^{p-4} \left(\mathrm{tr}( N) + (p-2)N_{nn} + O( \epsilon^{\frac{1}{2}})\right),
\end{aligned}
$$
where the implicit constant in $O( \epsilon^{\frac{1}{2}})$ can be controlled by $\delta, N, p, n$.
Hence, for sufficiently small $\epsilon$, we have
$ \Delta_p \left( \overline{a}^{+} Q^{+} \right) > 0$, in $ \Omega_v^{+} \cap B_1$. Similar computations show that $ \Delta_p \left( -\overline{a}^{-} Q^{-} \right) > 0$, in $ \Omega_v^{-} \cap B_1$. Thus, $v$ is a strict subsolution of \eqref{OVERDETERMINED} in $B_1$.

It remains to verify the free boundary condition. We claim that
$$ |\nabla v^+|^p-|\nabla v^-|^p>\lambda_+^p-\lambda_-^p. $$
Indeed,
$$ 
\begin{aligned}
|\nabla v^+|^p & =(\overline{a}^+)^p \left( 1+|\overline{M}x|^2+\epsilon^2 |\xi'|^2 + 4 \epsilon (Nx) \cdot \textbf{e}_n + 2 \epsilon (\overline{M} x) \cdot \xi' \right)^{\frac{p}{2}}\\
  & =(a^+)^p\left(1+p \epsilon  \left( s+ \frac12 |Mx|^2 + 2 (Nx) \cdot \textbf{e}_n \right)  + O(\epsilon^{\frac{3}{2}})\right).
\end{aligned}
$$
On $\{v=0\}$, we have $x_n=O(\epsilon^{\frac{1}{2}})$. Since $N\ee_n=N_{nn}\ee_n$, it follows that $(Nx) \cdot \textbf{e}_n = O(\epsilon^{\frac{1}{2}})$. 
Together with assumption $\|{M}\|  \leq  \delta \epsilon^{\frac{1}{2}}$, this results in
$$ |\nabla v^+|^p \geq (a^+)^p \left( 1 +p \epsilon \left(s-\delta^2 \right) \right), $$
provided $\epsilon$ is sufficiently small; depending on $s$, $t$, $N$, $\xi'$, and $\delta$.
Similarly,
$$ |\nabla v^-|^p \leq (a^-)^p \left( 1+p\epsilon \left(t+\delta^2 \right) \right).$$
Thus, according to \eqref{Ext34} and the condition $(a^+)^p-(a^-)^p=\lambda_+^p-\lambda_-^p$, the boundary condition holds.
\end{proof}


\subsection{Proof of Proposition \ref{Ext09}}

Let $V_{a^{+},M,\nu}$ be the quadratic polynomial defined in \eqref{QPNOT}. 
Given a continuous function $v$ in $B_1$, we define its $\epsilon$-linearization around the function $V_{a^{+},M,\nu}$ by
\begin{equation}
\label{Ext31}
\tilde{v}_{a^{+},M,\nu, \epsilon}(x):=
\begin{cases}
\dfrac{v(x)-a^+ P_{M,\nu}(x)}{a^+ \epsilon},  & \qquad x \in \overline{\Omega_v^+ \cap B_1}, \\
\\
\dfrac{v(x)-a^- P_{M,\nu}(x)}{a^- \epsilon},  & \qquad x \in \overline{\Omega_v^- \cap B_1}.
\end{cases}
\end{equation}
In what follows, we will typically drop the indices from $V$, $P$, $\tilde{v}$ whenever there is no possibility of confusion.

\medskip

\begin{proof}[Proof of Proposition \normalfont{\ref{Ext09}}]
The proof is divided into three steps.

\medskip

{\bf Step 1. Compactness.} Fix a universal radius $r_1$,  be specified later in Step 3. 
Arguing by contradiction, assume that there exist sequences $\epsilon_k \to 0$, and $\delta_k \to 0$, and also a sequence $u_k$ of solutions to \eqref{OVERDETERMINED} in $D=B_1$ such that
\begin{equation}
\label{Ext40}
|u_k-V_{a_k^{+}, M_k, \textbf{e}_n}|\leq \epsilon_k, \qquad \text{in} \quad B_1, \qquad \text{and} \qquad 0 \in \Gamma_{\mathrm{TP}}(u_k),
\end{equation}
with 
$$ \|M_k\|\leq \delta_k \epsilon_k^{\frac{1}{2}}, \quad M_k\textbf{e}_n=0,\quad \left(a_k^+ \right)^p- \left(a_k^- \right)^p= \lambda_+^p-\lambda_-^p, \quad a_k^{+} \in \left[L_0, L_1\right], $$
but $u_k$ does not satisfy the conclusion \eqref{Ext12}.

Let $\tilde{u}_k:=\tilde{u}_{a_k^{+}, M_k, \textbf{e}_n, \epsilon_k}$ be defined as in \eqref{Ext31}. Then, \eqref{Ext40} gives us
\begin{equation*}
|\tilde{u}_k|\leq C, \qquad \text{for} \quad x\in B_1. 
\end{equation*}
By Lemma \ref{Ext35}, the oscillation of $\tilde{u}$ decreases by a factor $1-c$ when passing from $B_1$ to $B_{\frac{1}{2}}$ (see e.g. \cite[proof of Lemma 7.13]{velichkov2019regularity}). 
This result can be iterated $m$-times, as long as $\epsilon_k(2(1-c))^m\leq \overline{\epsilon}$, yields a uniform H\"older modulus of continuity. 
Therefore, by the Arzel\'a-Ascoli theorem, a subsequence of $\tilde{u}_k$ converges uniformly in $B_{\frac{1}{2}}$ to a H\"older continuous function $u^*$.
Also, up to a subsequence, we may also assume that
$$ a_k^+ \to \overline{a}^+ \in \left[L_0,L_1 \right], $$
and hence
$$ a_k^- \to \overline{a}^-=\left( (\overline{a}^+)^p+\lambda_-^p-\lambda_+^p \right)^{\frac{1}{p}}. $$
Notice that $0 \in \Gamma_{\mathrm{TP}}(u_k)$ implies that $\tilde{u}_k(0)=0$; hence $u^*(0)=0$.

\medskip

{\bf Step 2. Limiting solution.} We now show that $u^*$ solves
\begin{equation}
\label{Ext42}
\begin{cases}
\mathcal{L}_p (u) = 0, & \qquad \text{in} \quad  B_{\frac{1}{2}} \cap \{x_n \neq 0\}, \\
a\partial_n^+ u-b\partial_n^- u=0, & \qquad \text{on} \quad B_{\frac{1}{2}} \cap \{x_n =0\},
\end{cases}
\end{equation}
where $a=\left(\overline{a}^+\right)^p$, $b=\left(\overline{a}^- \right)^p $, and
\begin{equation}
\label{Linearizedop}
\mathcal{L}_p (u) := \Delta u +(p-2)\partial_{nn} u.
\end{equation}

We first verify that $\mathcal{L}_p(u^*)=0$. 
Let $P(x)$ be a quadratic polynomial touching $u^*$ strictly from below at a point $\overline{x} \in B_{\frac{1}{2}}\cap \{x_n>0\}$. 
We need to show that at this point
$$ \mathcal{L}_p(P)=\Delta P + (p-2) \partial_{nn} P \leq 0. $$
Since $\tilde{u}_k \to u^*$, there exist points $x_k \in \Omega^+_{u_k} \cap B_{\frac{1}{2}}$ with $x_k \to \overline{x}$, and constants $c_k \to 0$ such that
\begin{equation}
\label{EEE1}
\tilde{u}_k(x_k)=P(x_k)+c_k,
\end{equation}
and
\begin{equation}
\label{EEE2}
\tilde{u}_k \geq P+c_k, \qquad \text{in a neighborhood of $x_k$.} 
\end{equation}
Recalling the definition of $\tilde{u}_k$, conditions \eqref{EEE1} and \eqref{EEE2} read
$$ u_k(x_k)=Q_k(x_k), $$
and
$$ u_k(x) \geq Q_k(x), \qquad \text{in a neighborhood of $x_k$,} $$
where
$$ Q_k(x)=\epsilon_k a_k^+ (P(x)+c_k)+ a_k^+ P_{M_k,\ee_n}(x). $$
Note that
\begin{equation*}
\nabla Q_k (x) = \epsilon_k a_k^+ \nabla P (x) + a_k^+ \left( \textbf{e}_n + M_k x \right),
\end{equation*}
thus by recalling $\|M_k\|=O(\epsilon_k^{\frac{1}{2}})$,
\begin{equation*}
\nabla Q_k(x_k) \neq 0, \qquad \text{for $k$ large.}
\end{equation*}

Since $u_k$ is $p$-harmonic, and $Q_k$ touches $u_k$ from below at $x_k$, and $\nabla Q_k(x_k) \neq 0$, by the equivalence of weak and viscosity solutions of $p$-harmonic functions, we get
$$ 
\begin{aligned}
0 &\geq \Delta_p Q_k(x_k) \\
& = \mathrm{div}\left( |\nabla Q_k(x_k)|^{p-2} \nabla Q_k(x_k) \right) \\
& = \left| \nabla Q_k(x_k) \right|^{p-2} \left( \Delta Q_k(x_k) + (p-2) \left( \frac{\nabla Q_k(x_k)}{|\nabla Q_k(x_k)|} \right)^T D^2 Q_k(x_k) \left( \frac{\nabla Q_k(x_k)}{|\nabla Q_k(x_k)|} \right) \right) \\
& = \left| \nabla Q_k(x_k) \right|^{p-2} \Bigg( \epsilon_k a_k^+ \Delta P (x_k) + a_k^+ \mathrm{tr}(M_k) \\
& \qquad \qquad \qquad  \qquad + \frac{p-2}{|\nabla Q_k(x_k)|^2} \left(\nabla Q_k(x_k) \right)^T \left( \epsilon_k a_k^+ D^2 P (x_k) + a_k^+ M_k \right) \nabla Q_k(x_k) \Bigg) \\
& = a_k^+ \left| \nabla Q_k(x_k) \right|^{p-2} \left( \epsilon_k \Delta P (x_k) + \frac{p-2}{|\nabla Q_k(x_k)|^2} \left(\nabla Q_k(x_k) \right)^T \left( \epsilon_k D^2 P (x_k) + M_k \right) \nabla Q_k(x_k) \right) \\
& =  \epsilon_k a_k^+ \left| \nabla Q_k(x_k) \right|^{p-2} \left( \Delta P (x_k) + \frac{p-2}{|\nabla Q_k(x_k)|^2} \left(\nabla Q_k(x_k) \right)^T D^2 P (x_k) \nabla Q_k(x_k) + O(\epsilon_k^{\frac{1}{2}})\right). \\
\end{aligned}
$$
Now, dividing both sides by $\epsilon_k$, and passing to the limit $k \to \infty$, and recalling that 
$$ \nabla Q_k(x_k) \to \overline{a}^+ \textbf{e}_n, $$
we conclude that 
$$ \Delta P(\overline{x}) + (p-2) \partial_{nn} P(\overline{x}) \leq 0. $$ 
The opposite inequality follows by considering test polynomials touching from above.  
Also, the same reasoning can be applied to prove $\mathcal{L}_p(u^*)=0$ in $B_{\frac{1}{2}}\cap \{x_n<0\}$. 

Now, we verify the transmission condition on $\{x_n=0\}$ in the viscosity sense. 
Let
$$ w(x):=A+sx_n^+-tx_n^-+x^T N x+\xi' \cdot x', $$
with $A \in \mathbb{R}$ and $N$ a matrix such that $\mathcal{L}_p (w)>0$, and 
$$ as-bt>0. $$ 
We can always assume, via a change of coordinate in $\bR^{n-1}$, that $N$ satisfies \eqref{NMatrixCon}. 
Suppose that $w$ touches $u^*$ strictly from below at a point $x_0=(x_0',0) \in B_{\frac{1}{2}}$. 
Define
$$ \overline{a}_k^+=a_k^+(1+\epsilon_k s), \qquad  \overline{a}_k^-=a_k^-(1+\epsilon_k t), $$
$$ \overline{M}_k=M_k+2\epsilon_k N, \qquad Q_k:=P_{\overline{M}_k,\textbf{e}_n}+\epsilon_k \xi' \cdot x' + A \epsilon_k, $$
and
$$ v_k:=\overline{a}_k^+ Q_k^+-\overline{a}_k^- Q_k^-. $$
Let $\tilde{v}_k$ be the $\epsilon$-linearization of $v_k$ around the function $V_{a^{+},M_k,\textbf{e}_n}$; then, it 
converges uniformly to $w$ on $B_{\frac{1}{2}}$. 
Since $\tilde{u}_k$ converges uniformly to $u^*$, and $w$ touches $u^*$ strictly from below at $x_0$, we conclude that for a sequence of constants $c_k \to 0$, and points $x_k \to x_0$, the function
$$ w_k(x):=v_k(x+\epsilon_k c_k \textbf{e}_n), $$
touches $u_k$ from below at $x_k$. 
Thus, Lemma \ref{Ext33} gives that $w_k$ is a strict subsolution of our free boundary problem provided that we first choose $\delta_k$ small enough to ensure that \eqref{Ext34} holds.
We reach a contradiction as we let $k \to \infty$; thus $u^*$ is a solution to the linearized problem \eqref{Ext42}. 

\medskip

{\bf Step 3. Contradiction.} 
Since $\tilde{u}_k$ converges uniformly to $u^*$, and $u^*$ enjoys the quadratic estimate of \cite[Theorem 3.2]{MR3218810} or \cite[Lemma 4.12]{fotouhi-bayrami2023}, we have 
\begin{equation}
\label{Ext43}
|\tilde{u}_k-(x' \cdot \nu'+ \tilde{s}x_n^+-\tilde{t}x_n^-)|\leq C r^{2}, \qquad \text{for} \quad x \in B_r, 
\end{equation}
where
$$ a \tilde{s}-b \tilde{t}=0, \qquad |\nu'|\leq C. $$
Define
$$ \overline{a}_k^+=a_k^+ \left(1+\epsilon_k \tilde{s} \,\right), \qquad \overline{a}_k^-=\left( (\overline{a}^+_k)^p+\lambda_-^p-\lambda_+^p \right)^{\frac{1}{p}}=a_k^- \left(1+\epsilon_k \tilde{t}\,\right)+o(\epsilon_k), $$
and
$$ \nu_k=\frac{\textbf{e}_n+\epsilon_k(\nu',0)}{\sqrt{1+\epsilon_k^2|\nu'|^2}}=\textbf{e}_n+\epsilon_k(\nu',0)+\epsilon_k^2 \tau, \qquad |\tau| \leq C. $$
We claim that $M_k$ admits a decomposition
\begin{equation}
\label{Ext44}
M_k=\overline{M}_k+D_k,
\end{equation}
as
$$ \overline{M}_k \nu_k=0, \qquad \mathrm{tr}(\overline{M}_k)=0, \qquad \|D_k\|\leq C\|M_k\| \, |\nu_k-\textbf{e}_n|; $$
hence $\|D_k\|=O(\epsilon_k^{\frac{3}{2}})$.

To prove the claim, we first decompose $M_k$ as  
$$ M_k =N_k+L_k, $$
such that $N_k \nu_k=0$, and
$$ \|L_k\| \leq C \|M_k \| \, |\nu_k-\textbf{e}_n|. $$
Since $\mathrm{tr}(M_k)=0$, we obtain 
$$ \mathrm{tr}(N_k)=O(\|L_k\|). $$
Then,  we can decompose $N_k$ further (here $\nu_k^{\perp}$ is a unit vector perpendicular to $\nu_k$):
$$ N_k=\overline{M}_k+z_k \left(\nu_k^{\perp} \otimes \nu_k^{\perp} \right), \qquad z_k=\mathrm{tr}(N_k), \qquad \overline{M}_k \nu_k=0, $$ 
so that $ \mathrm{tr}(\overline{M}_k)=0$,
and the claim \eqref{Ext44} is proved.

We now show that, for a universal radius $r=r_1$,
$$ |u_k - V_{\overline{a}_k^{+},\overline{M}_k,\nu_k}| \leq \epsilon_k r^{1+\alpha_1}, $$
say with $\alpha_1=\frac12$. 
This contradicts the fact that $u_k$ does not satisfy \eqref{Ext12}. 
Using \eqref{Ext43} and the definition of $\tilde{u}_k$, we get that in $\overline{\Omega_{u_k}^+ \cap B_r}$ (we can argue similarly in the negative part)
$$ |u_k-a_k^+P_{M,\textbf{e}_n}-\epsilon_ka_k^+ (x' \cdot \nu' + \tilde{s} x_n^+ - \tilde{t} x_n^-) | \leq 2 \epsilon_k C r^{2}. $$
Since $x_n \geq -3 \epsilon_k^{\frac{1}{2}}$ in this set, for $k$ sufficiently large,
$$ |u_k-\overline{a}_k^+P_{M,\textbf{e}_n}-\epsilon_ka_k^+ (x' \cdot \nu' + \tilde{s} x_n) | \leq 3 \epsilon_k C r^{2}, $$
which gives
$$  \left|u_k-\overline{a}_k^+ \left( x \cdot \nu_k + \frac{1}{2} x^T Mx \right) \right| \leq 4 \epsilon_k C r^{2}. $$
Finally, from \eqref{Ext44} we infer that
$$ |u_k - V_{\overline{a}_k^{+},\overline{M}_k,\nu_k}| \leq 5 C \epsilon_k r^{2}, $$
from which the desired bound follows for $r_1$ small enough.
\end{proof}


\section{Nonlinear dichotomy}
\label{Nonlinear dichotomy}

In this section, we establish the following dichotomy result for solutions of \eqref{OVERDETERMINED}. This proposition plays a key role in the proof of Theorem \ref{Ext02}.

\begin{proposition}[Nonlinear dichotomy]
\label{Ext07}
Let $u$ be a solution to \eqref{OVERDETERMINED} with $|\nabla u| \leq 1$ and assume that $0 \in \Gamma_{\mathrm{TP}}(u)$. 
Then, there exist universal positive constants $\epsilon_0$, $\delta_0$, $r_0$, $c_0$, $\alpha_0$ such that if
\begin{equation*}
|u(x)-P_{M,\nu}(x)| \leq \epsilon, \qquad \text{for} \quad x\in B_1,
\end{equation*}
with $\epsilon\leq \epsilon_0$, and
$$ \|M\| \leq \delta_0 \epsilon^{\frac{1}{2}}, $$
then one of the following two alternatives holds:
\begin{enumerate}[label=(\roman*)]
\item
$|\nabla u^{+}(0)|\leq 1-c_0 \epsilon$, or

\item
$|u-P_{\overline{M},\overline{\nu}}|\leq \epsilon r_0^{2+\alpha_0}$ in $B_{r_0}$ for some $\overline{M}$, $\overline{\nu}$ with $\|M-\overline{M}\|\leq C\epsilon$ and $C$ universal.

\end{enumerate}
\end{proposition}

\begin{remark}\label{rmk-iterate}
If alternative $(ii)$ of Proposition \ref{Ext07} holds, then the conclusion can be iterated. Indeed, the rescaling
$$ u_r(x):=\frac{u(rx)}{r}, \qquad r=r_0, $$
satisfies
$$ |u-P_{r\overline{M},\overline{\nu}}| \leq \epsilon_r:=\epsilon r^{1+\alpha_0}, \qquad \text{in} \quad B_1, $$
with
$$ \|r \overline{M} \| \leq \delta_0 \epsilon_r^{\frac{1}{2}}, \qquad \epsilon_r \leq \epsilon \leq \epsilon_0, $$
provided that $\epsilon_0$ is chosen small enough; see Remark \ref{LLREM1}.
\end{remark}

The proof of Nonlinear dichotomy, i.e. Proposition \ref{Ext07}, is very similar to that of Proposition \ref{Ext09}. 
However, before that, we need the following linear counterpart of the nonlinear dichotomy.

\begin{proposition}[Linear dichotomy]
\label{LinearDi}
Let $v$ solves
\begin{equation}
\label{Di-01}
\begin{cases}
\mathcal{L}_p (v) = 0, & \qquad \text{in} \quad B_{1} \cap \{x_n \neq 0\}, \\
\partial_n^+ v-b\partial_n^- v=0, & \qquad \text{on} \quad B_{1} \cap \{x_n =0\},
\end{cases}
\end{equation}
with $0<b_0 \leq b \leq b_1$ and $\|v\|_{L^{\infty}} \leq 1$. Moreover, suppose that 
\begin{equation*}
\partial_n v \leq 0.
\end{equation*}
Then, there exist universal constants $c_0, r_0 > 0$ such that either
\begin{enumerate}[label=(\roman*)]
\item $\partial_n^+ v(0) \leq - c_0$, or
\item $|v(x)-Q(x')| \leq \frac{1}{4} r_0^2$ in $B_{r_0}$,
\end{enumerate}
depending only on $x'$ and satisfying with 
\begin{equation*}
\mathcal{L}_p (Q) = \Delta Q = 0, \qquad \|Q\| \leq C,
\end{equation*}
for a universal constant $C$.
\end{proposition}

\begin{proof}
Let $r_0$ be given to be specified later. 
Assume by contradiction that there exist sequences  $c_k \to 0$, $b_k \in [b_0,b_1]$, and bounded monotone solutions $v_k$ of \eqref{Di-01} such that 
$$ \partial_n^+ v_k(0)>-c_k, $$
while alternative $(ii)$ fails for every $k$.
Passing to a subsequence, we may assume that $v_k$ converges uniformly on compacts and in the $C^{1,\alpha}$ norm from either side of $\{x_n \neq 0\}$ to a limiting solution $\overline{v}$ of
\begin{equation}
\label{Di-04}
\begin{cases}
\mathcal{L}_p (\overline{v}) = 0, & \qquad \text{in} \quad B_{\frac{3}{4}} \setminus \{x_n \neq 0\}, \\
\partial_n^+ \overline{v}-b\partial_n^- \overline{v} =0, & \qquad \text{on} \quad B_{\frac{3}{4}} \cap \{x_n =0\},
\end{cases}
\end{equation}
for some $b\in[b_0,b_1]$. Moreover, 
\begin{equation}
\label{Di-05}
\|\overline{v}\|_{L^{\infty}} \leq 1, \qquad \partial_n \overline{v} \leq 0, \qquad \partial_n^+ \overline{v}(0) = 0.
\end{equation}
Subtracting a suitable linear function depending only on $x'$, we may assume that
\begin{equation}
\label{Di-06}
\overline{v}(0)=0, \qquad \nabla \overline{v}(0) = 0.
\end{equation}
Indeed, by \eqref{Di-05} and the transmission condition in \eqref{Di-04}, we also have $\partial_n^-\overline v(0)=0$.

\medskip

{\bf Step 1.} We claim that
$$ D_{x'}^2 \overline{v} \leq C I_{x'}, \qquad \text{in} \quad B_{\frac{1}{4}}, $$
or, in other words that $\overline{v}$ is uniformly semiconcave in the $x'$-variable. 

It is enough to prove that, for every unit vector $\tau\perp \mathbf e_n$,
\begin{equation}
\label{Di-07}
\overline{v}_{\tau \tau}^h(x): =\frac{\overline{v}(x+h \tau)+\overline{v}(x-h \tau)-2\overline{v}(x)}{h^2} \leq C, \qquad \text{in} \quad B_{\frac{1}{4}},
\end{equation}
with a constant independent of $h>0$. 
We first establish the estimate
\begin{equation}
\label{Di-08}
\overline{v}_{\tau \tau}^h(x) \leq \frac{C}{x_n^2}, \qquad \text{in} \quad B_{\frac{1}{2}}.
\end{equation}
Indeed, by the interior $C^{2,\alpha}$ estimates applied in $B_{\frac{|x_n|}{2}}(x)$, we have
$$ | D^2 \overline{v}(x) | \leq \frac{C}{x_n^2} \| \overline{v} \|_{L^{\infty}} \leq \frac{C}{x_n^2}, \qquad \text{in} \quad B_{\frac{5}{8}} \setminus \{x_n=0\}. $$
Since we can write
$$ \overline{v}_{\tau \tau}^h(x_0)= \int_{-1}^{1} \overline{v}_{\tau \tau} (x_0+th \tau)(1-|t|) \, dt, $$
we immediately obtain \eqref{Di-08}.

Finally, note that $\overline v_{\tau\tau}^h$ satisfies the same transmission problem as $\overline v$. Combining \eqref{Di-08} with the weak $L^q$-Harnack inequality for small exponent $q>0$ (see for example, \cite[Theorem 4.8(2)]{MR1351007}), we deduce \eqref{Di-07}. 

\medskip

{\bf Step 2.} In this step, we wish to show that
\begin{equation}
\label{Di-09}
\| \overline{v} \|_{L^{\infty}(B_r)} \leq C r^2, \qquad r \leq \frac{1}{4},
\end{equation}
with $C$ universal. 

Let us consider
\begin{equation*}
\tilde{v}(x)=\frac{\overline{v}(rx)}{r^2}, \qquad x \in B_1.
\end{equation*}
Then, $\tilde{v}$ satisfies \eqref{Di-04} and \eqref{Di-05}. From \eqref{Di-06}, we also have
$$ \tilde{v}(0)=0, \qquad \nabla \tilde{v}(0)=0. $$
From Step 1, we deduce that
\begin{equation}
\label{Di-11}
\tilde{v}(x',0)\leq C|x'|^2, \qquad \text{in} \quad B_1 \cap \{x_n=0\}.
\end{equation}
Hence, by the monotonicity of $\tilde{v}$, we obtain
\begin{equation}
\label{Di-12}
\tilde{v}\leq C, \qquad \text{in} \quad B_1 \cap \{x_n>0\}.
\end{equation}
Next, we claim that 
\begin{equation}
\label{Di-13}
\tilde{v} \geq -LC, \qquad \text{in} \quad B_{\frac{1}{3}}\left(\frac{\textbf{e}_n}{2} \right),
\end{equation}
for some large constant $L$ to be specified later. 
Suppose by contradiction that this does not hold at some point $z\in B_{\frac{1}{3}}\left(\frac{\textbf{e}_n}{2}\right)$.
Let $z_1\in \{x_n=d\}\cap\{|x'|=\frac13\}$ be arbitrary, with $d$ sufficiently small.
We then construct a sequence of points $z_1, \cdots, z_k$ such that $z_k=z$ and
\[
|z_{i+1}-z_i| \le d2^{i-1}, \qquad i=1, \cdots, k-1.
\]
Note that $k$, i.e. the number of points, depends only on $d$ and is of order $|\log d|$.
Then, by virtue of  Harnack inequality for $C-\tilde{v} \geq 0 $ in $B_{d2^{i}}(z_i)$, there is a constant $c_0<1$ such that 
\[
C-\tilde{v}(z_{i}) \ge c_0 (C-\tilde{v}(z_{i+1})).
\]
This implies that
$$ C-\tilde{v} \geq (1+L)C c_0^k, \qquad \text{on} \quad x_n=d, \quad |x'| \leq \frac{1}{3}, $$
with $k$ depends only on $d$. 
In particular, if 
$$ c_0^k \geq \frac{2}{1+L}, $$
then
\begin{equation}
\label{Di-15}
\tilde{v} \leq -C, \qquad \text{on} \quad x_n=d, \quad |x'| \leq \frac{1}{3}.
\end{equation}
We now compare $\tilde{v}$ with the explicit barrier
$$ \phi(x):=10C|x'|^2-Ax_n^2-x_n, $$
where $A=A(C,p, n)$ is chosen so that $ \mathcal{L}_p ( \phi ) \leq 0 $.
Define
$$ R:= \left\{0<x_n<d, \quad |x'|<\frac{1}{3} \right\}. $$
We show that for $d$ small enough (and hence $L$ large enough), we have  
\begin{equation}
\label{Di-16}
\tilde{v} \leq \phi, \qquad \text{on} \quad \partial R.
\end{equation}
Thus, we conclude that the inequality holds in $R$ and we reach a contradiction because
$$ 0=\partial_n \tilde{v} (0) \leq \partial_n \phi(0)=-1. $$
Now, we verify \eqref{Di-16}. On $\{x_n=d\}$, this follows from \eqref{Di-15} if $d$ is chosen small depending on $C$ and $A$. Similarly, on $\{x_n=0\}$, the desired bound follows immediately from \eqref{Di-11}. Finally, in the set 
$$ \left\{0<x_n<d, \quad |x'|=\frac{1}{3} \right\}, $$
we use \eqref{Di-12} and again we obtain \eqref{Di-16} for $d$ sufficiently small. In conclusion, the claim \eqref{Di-13} holds.

Since $\tilde{v}$ is decreasing in the $\textbf{e}_n$-direction, we obtain 
$$ \tilde{v} \geq - C, \qquad \text{in} \quad B_{\frac{1}{3}}, $$
for some  $C$. Now, recalling that $\tilde{v}(0)=0$, and applying  Harnack inequality for $\tilde{v}+C$, we get 
$$ |\tilde{v}| \leq C, \qquad \text{in} \quad B_{\frac{1}{4}}, $$
which after rescaling gives the desired claim \eqref{Di-09}.

\medskip

{\bf Step 3.} We prove that if $\overline{v}$ solves \eqref{Di-04} and satisfies \eqref{Di-05}, \eqref{Di-06}, and \eqref{Di-09}, then there exists a universal $r_0$ such that alternative $(ii)$ holds for $\overline{v}$ with right-hand side $\frac{1}{8} r_0^2$  in place of  $\frac{1}{4} r_0^2$. 

This claim follows by compactness and relies on the classification of global solutions to $\mathcal{L}_p$, namely Lemma \ref{C-G-S}.
Assume by contradiction that there exists a sequence of $\delta_k \to 0$, and solutions $\overline{v}_k$ to a sequence of problems $\eqref{Di-04}_k$ (satisfying the same properties as $\overline{v}$) for which the alternative $(ii)$ fails in the ball of radius $\delta_k$. 
Denote the quadratic rescalings by
$$ w_k(x)=\frac{\overline{v}_k(\delta_kx)}{\delta_k^2}. $$
Then, up to extracting a subsequence, $w_k$ converges uniformly on compacts to a global solution $U$ to the limiting transmission problem
\begin{equation*}
\begin{cases}
\mathcal{L}_p (U) = 0, & \qquad \text{in} \quad \mathbb{R}^n \cap \{x_n \neq 0\}, \\
\partial_n^+ U=g \partial_n^- U, & \qquad \text{on} \quad \{x_n =0\},
\end{cases}
\end{equation*}
with $0<b_0 \leq g \leq b_1$. 
The convergence holds in the $C^{1,\alpha}$ norm up to ${x_n=0}$ from either side and in the $C^{2,\alpha}$ norm in the interior. 
Clearly the global solution $U$ also satisfies \eqref{Di-05}, \eqref{Di-06}, and \eqref{Di-09}. 

Lemma \ref{C-G-S} then yields $U=Q(x')$ where $Q$ is a quadratic polynomial depending only on $x'$ and satisfying $\mathcal{L}_p (Q) = \Delta Q =0$. 
Therefore, for $k$ large, $\overline{v}_k$ satisfies the alternative $(ii)$ in $B_{\delta_k}$.
The contradiction establishes Step 3. 

\medskip

By Step 3, $\overline{v}$ satisfies alternative $(ii)$ with right-hand side $\frac{1}{8}r_0^2$ for some quadratic polynomial $\overline{Q}(x')$ with $\mathcal{L}_p (\overline{Q}) = \Delta \overline{Q} =0$. 
As above, this means that the $v_k$'s satisfy the alternative $(ii)$ for all $k$ large.
\end{proof}

\begin{proof}[Proof of Proposition \normalfont{\ref{Ext07}}]
We divide the proof into two steps.

\medskip
{\bf Step 1. Compactness and the limiting solution.} Fix a universal constant $r_0$, to be chosen precise later. 
Toward a contradiction, assume  that there are  sequences $\epsilon_k \to 0$, $\delta_k \to 0$ and a sequence $u_k$ of solutions to \eqref{OVERDETERMINED} in $D=B_1$ such that
\begin{equation*}
|u_k-P_{M_k, \textbf{e}_n}|\leq \epsilon_k, \qquad \text{in} \quad B_1,  \qquad  \text{and} \qquad 0 \in \Gamma_{\mathrm{TP}}(u_k),
\end{equation*}
with 
$$ \|M_k\|\leq \delta_k \epsilon_k^{\frac{1}{2}}, \qquad M_k\textbf{e}_n=0,\qquad \mathrm{tr}(M_k)=0, $$
but $u_k$ does not satisfy either alternative $(i)$ (for some small constant $c_0'$ to be specified later) or alternative $(ii)$. 

Let $\tilde{u}_k:=\tilde{u}_{1, M_k, \textbf{e}_n, \epsilon_k}$ be defined as in \eqref{Ext31}. 
By Steps 1 and 2 in the proof of Proposition \ref{Ext09}, after passing to a subsequence, $\tilde u_k$ converges uniformly in $B_{\frac12}$ to a H\"older continuous function $u^*$ with $u^*(0)=0$, $\|u^*\|_{L^{\infty}} \leq C$, and $u^*$ solves the transmission problem
\begin{equation*}
\begin{cases}
\mathcal{L}_p (u^*) = 0, & \qquad \text{in} \quad B_{\frac{1}{2}} \cap \{x_n \neq 0\}, \\
\partial_n^+ {u^*}-b\partial_n^- {u^*}=0, & \qquad \text{on} \quad B_{\frac{1}{2}} \cap \{x_n =0\},
\end{cases}
\end{equation*}
with $b>0$ bounded by a universal constant. (Indeed, in Proposition \ref{Ext09} we can choose $L_0=\frac{1}{2}$ and $L_1=2$ due to the condition $|\nabla u|\le 1$.)

Next, we claim that $u^*$ is monotone decreasing in the $\textbf{e}_n$-direction, namely
\begin{equation}
\label{Ext47}
\partial_n {u^*} \leq 0.
\end{equation}
To prove this, recall that $\tilde u_k$ satisfies
\begin{equation}
\label{Ext47.1}
 \sum_{i,j=1}^n \theta_{ij}\partial_{x_ix_j} \tilde{u}_k +\sum_{i=1}^n \beta_i \partial_{x_i}\tilde{u}_k =\frac{2-p}{\epsilon_k|\nabla u_k|^2}x^T M_k^2\nabla u_k, \qquad \text{in} \quad \Omega_{u_k}^+\cap\{|\nabla u_k|\ne0\}. 
 \end{equation}
where $\theta_{ij}=\delta_{ij} + (p-2)|\nabla u_k|^{-2} \partial_{x_i}u_k \partial_{x_j}u_k $ and
$\beta_i= (p-2)|\nabla u_k|^{-2} \sum_{j=1}^n {M_{k}}_{_{ij}}\partial_{x_j}u_k$.

Let $\bar x\in\{x_n>0\}$. 
Since $\Gamma_{\mathrm{TP}}(u_k) \subset \{|x_n| \leq \epsilon_k^{\frac{1}{2}} \}$, for sufficiently large $k$ we have
$$ B:=B_{\frac{\overline{x}_n}{2}}(\overline{x}) \subset \Omega_{u_k}^+. $$
We first show that
\begin{equation}
\label{Ext47.2}
 |\nabla \tilde{u}_k | \leq C, \qquad \text{in} \quad B, 
  \end{equation}
where constant $C$ here depends only on $x_n$.
Inside $\{|\nabla u_k|\ge \frac14\}\cap B$, we get $|\nabla \tilde{u}_k | \le C$ from  \eqref{Ext47.1} for a universal constant $C$.
Thus, 
\[
|\nabla u_k| \ge 1- \delta_k\epsilon_k^{\frac12}-C\epsilon_k\ge \frac12,
\] 
for sufficiently large $k$. 
This along with the continuity of $\nabla u_k$ inside $\Omega_{u_k}^+$ necessitates that $|\nabla u_k| \ge \frac12$ for sufficiently large $k$. This proves \eqref{Ext47.2}

Hence, 
$$ \nabla u_k(\overline{x})= \nabla ( P_{M_k,\textbf{e}_n}+\epsilon_k \tilde{u}_k ) (\overline{x}) = \textbf{e}_n + M_k\overline{x} + \epsilon_k \nabla \tilde{u}_k (\overline{x}). $$
Using that $|\nabla u_k(\overline{x})|^2 \leq 1$ and $M_k \textbf{e}_n=0$, we get
$$ 1 \geq 1 + |M_k\overline{x}|^2+2 \epsilon_k \partial_n \tilde{u}_k (\overline{x}) + O(\epsilon_k^{\frac{3}{2}}), $$
where the constant in $O(\epsilon_k^{\frac{3}{2}})$ depends on $\overline{x}_n$. 
By $C^{1,\alpha}$ estimates of \eqref{Ext47.1} in the ball $B$,  $\tilde{u}_k$ converges to $u^*$ in $C^1$.
 Passing to the limit, as $k \to \infty$, we get that
$$ \partial_n u^*(\overline{x}) \leq 0. $$
By continuity, since $u^*$ is $C^{1,\alpha}$ up to $\{x_n=0\}$, we conclude that
$$ \partial_n^+ {u^*} \leq 0, \qquad \text{on} \quad B_{\frac12} \cap \{x_n = 0 \}. $$
We can argue similarly to show that  $\partial_n {u^*}\le0$ in $B_{\frac12} \cap \{x_n \le 0 \}$.

\medskip
{\bf Step 2. Contradiction.} 
According to Proposition \ref{LinearDi} (Linear dichotomy), since \eqref{Ext47} holds and $\|u^*\|_{L^{\infty}} \le 1$, there exist universal constants $c_0>0$ and $r_0 > 0$ such that either of the following conditions is satisfied:
\begin{enumerate}[label=$(\roman*)$]
\item 
$\partial_n^+ u^* (0) \leq - c_0$, or
\item
$ |u^*(x)-Q(x')| \leq \frac{1}{4}r_0^2 $ in $B_{r_0}$, where $Q(x')$ is a quadratic polynomial in the $x'$-direction with $\|Q\|\leq C$ universal and $\mathcal{L}_p (Q) = \Delta Q =0$. 
\end{enumerate}

If $(i)$ above is satisfied, by the quadratic estimate for $u^*$ (see \cite[Theorem 3.2]{MR3218810} or \cite[Lemma 4.12]{fotouhi-bayrami2023}), we obtain that
$$ | u^*(x)- (\xi' \cdot x' + s x_n^+-tx_n^-) | \leq C r^{2}, \qquad \text{for} \quad x\in B_r, $$
with $|s|, |t|, |\xi'| \leq C$, and
$$ s-bt=0, \qquad s \leq -c_0. $$
Using the convergence of $\tilde{u}_k$ to $u^*$ and the definition of $\tilde{u}_k$, we immediately get that for $k$ large,
$$ | u_k(x) - \left( P_{M_k,\textbf{e}_n}(x)+\epsilon_k \left( \xi' \cdot x' + s x_n^+- t x_n^- \right) \right) | \leq C \epsilon_k r^{2}, \qquad \text{for} \quad x\in B_r. $$
Arguing as in Step 3 of the proof of Proposition \ref{Ext09}, we conclude that for $r \leq r_1$,
$$ |u_k - V_{\overline{a}_k^{+},\overline{M}_k,\nu_k}| \leq \epsilon_k r^{1+\alpha_1}, \qquad \text{with} \quad \overline{a}^+ = 1+ \epsilon_k s. $$
Thus, after rescaling (see Remark \ref{LLREM1}), Corollary \ref{Ext13} yields
$$ | \nabla u_k^+(0) | \leq 1+ s \epsilon_k + C \epsilon_k r^{1+\alpha_1} \leq 1- \frac{1}{2} c_0 \epsilon_k, $$
provided that $r$ is chosen small, depending only on universal constants. 
Thus, $u_k$ satisfies alternative $(i)$ in Proposition \ref{Ext07} for $c_0'=\frac{c_0}2$ and we arrive at a contradiction. 

\medskip

Assume now that condition $(ii)$ above holds. 
Since $u^*(0)=0$, we may write
$$ Q(x')=\xi' \cdot x' + \frac{1}{2} x^T N x, \qquad N \textbf{e}_n=0, \qquad \mathrm{tr}(N)=0. $$
Now using the convergence of the $\tilde{u}_k$ to $u^*$ and the definition of $\tilde{u}_k$, we get
\begin{equation}
\label{Ext49}
\left|u_k(x)- \left( x_n + \epsilon_k \xi' \cdot x' + \frac{1}{2} x^T N_k^* x \right) \right| \leq \frac{1}{4} \epsilon_k r_0^2,
\end{equation}
with 
$$ N_k^*=M_k+\epsilon_k N, \qquad N_k^* \textbf{e}_n=0, \qquad \mathrm{tr}(N_k^*)=0, $$
where we have used
$\mathrm{tr}(M_k)=0$ and $\mathrm{tr}(N)=0$.
Finally, as before, denote
$$ \nu_k= \frac{\textbf{e}_n+\epsilon_k(\xi',0)}{\sqrt{1+\epsilon_k^2 |\xi'|^2}}=\textbf{e}_n+\epsilon_k(\xi',0)+\epsilon_k^2 \tau, \qquad |\tau| \leq C. $$
We argue as in Step 3 of Proposition \ref{Ext09}, and decompose
$$ N_k^*=\overline{M}_k+D_k, \qquad \overline{M}_k \nu_k=0, \qquad \|D_k\|=o(\epsilon_k), $$
with $ \mathrm{tr}(\overline{M}_k)=0 $.
Thus, \eqref{Ext49} yields for $k$ large
$$ |u_k-P_{\overline{M}_k,\nu_k}| \leq \frac{1}{2} \epsilon_k r_0^2 \leq \epsilon_k r_0^{2+\alpha_0}, $$
for $\alpha_0$ small enough, and again we have reached a contradiction.
\end{proof}


\section{Classification of Lipschitz Global Solutions}
\label{Main-theorems}

The objective of this section is to prove Theorem \ref{Ext02}, a Liouville type theorem for the Lipschitz global viscosity solutions of the two-phase Bernoulli problem \eqref{OVERDETERMINED}.
As a first step, we establish the following lemma showing that, at some (possibly small) scale, the solution can be approximated by a two-plane solution whose slopes attain the maximal values allowed by the positive and negative phases.
Combined with Proposition \ref{Ext07}, this yields a reverse improvement of flatness argument, from which we deduce that the same two-plane solution approximates $u$ at all sufficiently large scales.

Let $u$ be a Lipschitz global viscosity solution to \eqref{OVERDETERMINED} in $D=\mathbb{R}^n$. Define
\begin{equation}
\label{Ext16}
\alpha:=\sup_{\Omega_u^+} |\nabla u|, \qquad \beta:=\sup_{\Omega_u^-} |\nabla u|.
\end{equation}
An initial flatness condition for $u$ can be obtained by the following lemma.

\begin{lemma}
\label{Ext17}
Let $u$, $\alpha$ and $\beta$ be as above.  
Then,
$$ \alpha^p-\beta^p=\lambda_+^p-\lambda_-^p  \qquad  \text{and} \qquad  \alpha\ge \lambda_+, \quad  \beta\ge\lambda_-. $$
Moreover, if $\Gamma^{\mathrm{int}}_{\mathrm{TP}}(u) \neq \emptyset$, then for every fixed $R>1$ there exists a sequence of rescalings $u_k(x):=\frac{u(x_k+d_kx)}{d_k}$ such that 
$$ u_k \to \alpha x_n^+-\beta x_n^-, $$
uniformly on $B_R$, as $k \to + \infty$.
\end{lemma}

\begin{proof}
Given $\epsilon>0$, choose a point $x_{\epsilon} \in \Omega_u^+=\{x \in \mathbb{R}^n \, : \, u(x)>0\}$ such that
$$ |\nabla u(x_{\epsilon})| \geq \alpha-\epsilon. $$
Let $d_\epsilon$ be the distance from $x_{\epsilon}$ to $\partial \Omega_u^+$ and consider the rescaling
$$ u_{\epsilon}(x):=\frac{u(x_{\epsilon}+d_\epsilon x)}{d_\epsilon}. $$
These rescalings still satisfy \eqref{OVERDETERMINED}, and $B_1 \subset \Omega_{u_{\epsilon}}^+$ is tangent to the free boundary $\partial \Omega_{u_{\epsilon}}^+=\partial \{x \in \mathbb{R}^n \, : \, u_{\epsilon}(x) >0\}$. 
After a suitable rotation, we may furthermore assume that
$$ \nabla u_{\epsilon}(0)=t_{\epsilon} \textbf{e}_n, \qquad \alpha \geq t_{\epsilon} \geq \alpha-\epsilon. $$
Since $u_{\epsilon}$ is uniformly Lipschitz, we may extract a subsequence such that
$$ u_{\epsilon} \to \overline{u}, $$ 
uniformly on compacts in $\mathbb{R}^n$, as $\epsilon$ goes to zero. In addition, the convergence is in $C^{1,\eta}(B_{\frac{1}{2}})$ due to $\Delta_p u_\epsilon=0$ in $B_1$.
Consequently, for every fixed $R>1$, the limit $\overline u$ is a solution of \eqref{OVERDETERMINED} in $B_R$ and satisfies
$$ \partial_n \overline{u}(0)=\alpha, \qquad \text{and} \qquad |\nabla \overline{u}| \leq \alpha \qquad \text{in} \quad \Omega_{\overline{u}}^{+}, $$
while $B_1\subset \Omega_{\overline u}^+$ remains tangent to $\partial\Omega_{\overline u}^+$.
Hence, we have $\Delta_p \overline{u} = 0$ in $B_1$ and $\overline{u} \geq 0$ in $B_1$, but it is not identically zero because $\partial_n \overline{u} (0)=\alpha>0$. Therefore $\overline{u}>0$ in $B_1$. 
Since $\overline{u}$ is $p$-harmonic, it is locally $C^{1,\eta}$ and we can choose a small $\delta>0$ such that $\partial_n \overline{u} > \frac{\alpha}{2}$ in $B_{\delta}$.
In particular, the function $v=\partial_n \overline{u}$ satisfies the uniform elliptic equation 
\begin{equation}
\label{Ext18}
\sum_{i,j=1}^n \partial_i(a_{ij}\partial_j v) =0, \qquad a_{ij} = |\nabla \overline{u}|^{p-2}\left( \delta_{ij} + (p-2) \frac{\partial_i \overline{u} \partial_j \overline{u}}{|\nabla \overline{u}|^2}\right),
\end{equation}
in $B_{\delta}$.
Hence by the strong maximum principle
$$ \partial_n \overline{u} \equiv \alpha, \qquad \text{in} \quad \mathcal{O}, $$
where $\mathcal{O}$ is the connected component on $\Omega_{\overline{u}}^{+}$ containing $B_1$. 
Therefore,
$$ \overline{u}(x)=\alpha x_n+\mathrm{const}, \qquad \text{in} \quad \mathcal{O}. $$
Moreover, $\overline{u}=0$ at the point $-e_n$, where the unit ball is tangent to $\partial \Omega_{\overline{u}}^{+}$. Hence,
\begin{equation*}
\overline{u}(x)=\alpha(x_n+1), \qquad \text{in} \quad  \{x_n>-1\}.
\end{equation*}
On the other hand, $\overline{u}$ is also a viscosity solution to \eqref{OVERDETERMINED}  and we must have 
$$\alpha \ge \lambda_+.$$
Now recall \eqref{Ext16}. Then, $|\nabla \overline{u}| \leq \beta $ in $\Omega_{\overline{u}}^{-}$ shows that the function
$$ H(x):=\alpha(x_n+1)^+-\beta(x_n+1)^-, $$
touches $\overline{u}$ from below on $ \{x_n=-1\}$. 
We distinguish here two cases:

\medskip
\noindent{\bf Case $(i)$:} $\{x_n=-1\} \cap \Gamma_{\mathrm{OP}}(\overline{u}) \ne \emptyset$. \\
From the definition of viscosity solution (see \cite[Lemma 3.1]{fotouhi-bayrami2023}), we obtain that 
\[
\alpha = |\nabla H^+| \le \lambda_+.
\]
So, we have $\alpha =\lambda_+$ in this case.

\medskip
\noindent{\bf Case $(ii)$:} $\{x_n=-1\} \subseteq \Gamma_{\mathrm{TP}}(\overline{u})$.\\
Again by the definition of viscosity solution, we must have
\[
 \alpha^p- \beta^p = |\nabla H^+|^p-|\nabla H^-|^p \le \lambda_+^p-\lambda_-^p.
\]

\medskip

Similarly, we can argue as above, starting with points $x_{\epsilon} \in \{x \in \mathbb{R}^n \, : \, u(x)<0\}$ where $\nabla u(x_{\epsilon})>\beta-\epsilon$, to get $\beta \ge \lambda_-$, as well as one of the following cases hold: 
$$\beta = \lambda_-  \qquad \text{or} \qquad \alpha^p-\beta^p \ge \lambda_+^p-\lambda_-^p.$$ 
It follows that
$$ \alpha^p-\beta^p=\lambda_+^p-\lambda_-^p. $$

In order to prove  the second part of the lemma, notice that if either $\alpha>\lambda_+$ or $\beta > \lambda_-$, we are in situation Case $(ii)$, and hence $\Omega^-_{\overline{u}}=\{x_n<-1\}$. 
Since $H(x)$ touches $\overline{u}(x)$ from below on $ \{x_n=-1\}$, the Hopf boundary lemma yields
$$ \overline{u}(x)=\alpha(x_n+1)^+-\beta(x_n+1)^-. $$

Finally, consider the case $\alpha=\lambda_+$ and $\beta=\lambda_-$, and assume that $0\in \Gamma^{\mathrm{int}}_{\mathrm{TP}}(u)$.
Let 
\[
\ell = \limsup_{\Omega_u^+\ni x \to 0}|\nabla u(x)| \le \alpha=\lambda_+,
\]
and choose a maximizing sequence $y_k\to0$ such that $|\nabla u(y_k)|\to\ell$ and $u(y_k)>0$.
Repeating the argument from the first part of the proof, we conclude that $ \ell \geq \lambda_+$ and $\frac{u(y_k+d_k x)}{d_k} \to \tilde u(x)$, up to a subsequence. 
Also, $\tilde u(x) = \lambda_+ (x_n+1)$ in $\{x_n>-1\}$.
Since $0\in \Gamma^{\mathrm{int}}_{\mathrm{TP}}(u)$ and $|B_r\cap\{u=0\}|=0$ for some $r>0$, so $\{x_n=-1\} \subseteq \Gamma_{\mathrm{TP}}(\tilde{u})$ and 
\[
 \tilde{u}(x)=\lambda_+(x_n+1)^+-\lambda_-(x_n+1)^-.
\]
\end{proof}

\begin{remark}
From now on, for the sake of convenience, we assume that 
$$\alpha=\beta=1.$$ 
This is not a restrictive assumption as one can replace $u$ by $\alpha^{-1} u^+ - \beta^{-1} u^-$. In this case, we have $\lambda_+ = \lambda_- \le 1$.
\end{remark}

\begin{lemma}
\label{Ext20}
Assume $|\nabla u| \leq 1$, $0 \in \Gamma_{\mathrm{TP}}(u)$ and $t\textbf{e}_n \in \Omega_u^+$ for all $t \in (0,1]$. If 
$$ \partial_nu(t\textbf{e}_n) \geq 1-\eta, \qquad \forall t \in (0,1], $$
for some $\eta>0$, then
$$ |u(x)- x_n| \leq \epsilon(\eta), \qquad \text{in} \quad B_2, $$
with $\epsilon(\eta) \to 0$ as $\eta \to 0$. 
\end{lemma}

\begin{proof}
This follows immediately by a compactness argument, along the same lines of arguments as the proof of Lemma \ref{Ext17}. 
Assume by contradiction that there exists a sequence $\eta_k \to 0$, and a sequence of equi-Lipschitz solutions $u_k$ such that
\begin{equation}
\label{Ext21}
\partial_nu_k(t\textbf{e}_n) \geq 1-\eta_k, \qquad \forall t \in (0,1], \, t\textbf{e}_n \in \Omega_{u_k}^+,
\end{equation}
but
\begin{equation}
\label{Ext22}
|u_k(x)-x_n|>\delta, \qquad \text{at some point $x\in B_2$ for some fixed $\delta>0$}.
\end{equation}
Since $u_k(0)=0$ and by \eqref{Ext21}, we obtain
$$ u_k(\textbf{e}_n) \geq 1 - \eta_k. $$
Using the Lipschitz bound $|\nabla u_k| \leq 1$, it follows that
$$ u_k>0, \qquad \text{in} \quad B_{1-\eta_k}(\textbf{e}_n). $$
We can now argue as in the previous lemma and extract a subsequence that converges uniformly on $\overline{B_2}$ to $\overline{u}=  x_n^+- x_n^-=x_n$, which contradicts \eqref{Ext22}.
\end{proof}

We are now ready to provide the proof of Theorem \ref{Ext02}.

\begin{proof}[Proof of Theorem \normalfont{\ref{Ext02}}]
According to Lemma \ref{Ext17}, after a translation, a rotation, and a rescaling, we can assume that (recall also the assumption $\alpha=\beta=1$)
\begin{equation*}
0 \in \Gamma^{\mathrm{int}}_{\mathrm{TP}}(u), \qquad \text{and} \qquad |u- x_n|\leq \epsilon' \quad \text{in} \quad B_1,
\end{equation*}
where $\epsilon'>0$ will be chosen sufficiently small.  
Let $\epsilon_0$ be a sufficiently small universal constant so that the conclusions of Proposition \ref{Ext07}, Proposition \ref{Ext09}, and Corollary \ref{Ext13} hold for all $\epsilon\leq \epsilon_0$. If $\epsilon' \leq \epsilon_0$, then Corollary \ref{Ext13} implies that $ \Gamma^{\mathrm{int}}_{\mathrm{TP}}(u) \cap B_{\frac{1}{2}}$ is locally $C^{1,\eta}$ smooth and $u \in C^{1,\eta}\left( \overline{\Omega_u^+ \cap B_{\frac{1}{2}} }\right)$ . 

\medskip
{\bf Step 1. Propagation of the lower bound for $\partial_n u$.}
By denoting $h:=u-x_n \in C^{1,\eta}\left( \overline{ \Omega_u^+ \cap B_{\frac{1}{2}} } \right)$, we easily achieve the following uniformly elliptic equation 
\begin{equation}
\label{Ext23.5}
\sum_{i,j=1}^n  \theta_{ij}  \partial_{x_ix_j} h = 0, \qquad\text{in} \quad \{\nabla u \ne0\}, 
\end{equation}
with $C^{0,\eta}$ coefficients $\theta_{ij}=\delta_{ij} + (p-2)|\nabla u|^{-2} \partial_{x_i}u \partial_{x_j}u $. 
Let $B_r$, $r\le \frac12$, be the largest ball such that $\nabla u\ne0$ inside $\Omega_u^+ \cap B_r$; then, the flatness of the free boundary implies that $|u_r-x_n| \le C\epsilon'$ for a universal constant $C$. 
So, $h_r=u_r-x_n$ satisfies \eqref{Ext23.5} in $\Omega_{u_r}^+\cap B_1$, and the  $C^{1,\eta}$ estimate of $h_r$  yields that  
\[
|\partial_n u - 1| \le C\epsilon', \qquad \text{in} \quad \Omega_u^+ \cap B_r,
\]
for a universal constant $C$. 
If $\epsilon_0$ is sufficiently small, we can conclude that   $\nabla u \ne 0$ in $\Omega_u^+ \cap B_{\frac{1}{2}}$.
Using again $C^{1,\eta}$ estimate of $h$ in \eqref{Ext23.5}, we  deduce that
\begin{equation}
\label{Ext24}
\partial_n u(t\textbf{e}_n) \geq 1 - C \epsilon', \qquad \forall t \in \left(0,\frac{1}{2} \right],
\end{equation}
and consequently
\begin{equation}
\label{Ext25}
\partial_n u^+(0) \geq 1-C\epsilon'.
\end{equation}

Now let $\epsilon(\eta)$ be as defined in Lemma \ref{Ext20}, and choose a  universal constant $\overline{\eta}$ small enough such that $ \epsilon(\overline{\eta}) \leq \epsilon_0$. 
We claim that the following inequality holds
\begin{equation}
\label{ExtIMC}
\partial_n u(t\textbf{e}_n) \geq 1- \overline{\eta}, \qquad \text{for all} \quad t>0,
\end{equation} 
provided that $\epsilon'$ is chosen sufficiently small.

\medskip
{\bf Step 2. Consequences of  the claim \eqref{ExtIMC}.}
Assume for the moment that \eqref{ExtIMC} holds, Lemma \ref{Ext20} gives 
$$ |u-x_n|\leq \epsilon(\overline{\eta}) R \leq \epsilon_0 R, \qquad \text{in} \quad B_{R}, $$
for every $R>0$. 
This, along with the improvement of flatness \cite[Theorem 4.1]{fotouhi-bayrami2023}, implies that 
\[
|u-U_{{a}^{+},\overline{\nu}}| \leq \epsilon_0 \rho^{1+\gamma}R, \qquad \text{in} \quad B_{\rho R},
\]
for some $a^+=a^+(R)$, $\overline{\nu}=\overline{\nu}(R)$, and universal constants $\rho, \gamma>0$. Then, for all $R$
\[
\dist(u, \mathcal{U}_{\lambda_+,\lambda_-}) \le \epsilon_0 \rho^{\gamma}R,
\]
where the distance is measured in $L^\infty(B_R)$ and $\mathcal{U}_{\lambda_+,\lambda_-}$ is the family of all two-plane solutions. 
Repeat this argument to find out $\dist(u, \mathcal{U}_{\lambda_+,\lambda_-})\to 0$, for a fixed $R$. 
So, $u$ must be a two-plane solution which concludes the proof of Theorem \ref{Ext02}.

\medskip
{\bf Step 3. Proof of the claim \eqref{ExtIMC}.}
Let $\epsilon'':=C\epsilon'$, where $C$ is the constant appearing in \eqref{Ext24}. 
We will eventually choose $\epsilon''$ (and therefore $\epsilon'$) later, so that $\epsilon'' \ll \overline{\eta}$.
Let $\overline{t}$ be the first value of $t$ for which \eqref{ExtIMC} fails, and assume without loss of generality after a rescaling that $\overline{t}=1$. In other words,  
\begin{equation}
\label{Ext26}
\partial_nu(\textbf{e}_n)=1-\overline{\eta}, \qquad \text{and} \qquad \partial_nu(t\textbf{e}_n) \geq 1-\overline{\eta}, \quad \forall t \in [0,1]. 
\end{equation} 
Notice that after rescaling, \eqref{Ext25} is still satisfied, i.e.
$$ \partial_n u^+(0)\ge 1-\epsilon''. $$
Applying Lemma \ref{Ext20} together with \eqref{Ext26}, we conclude that
\begin{equation*}
|u- x_n|\leq \epsilon(\overline{\eta}) \leq \epsilon_0, \qquad \text{in} \quad B_2. 
\end{equation*}
Hence Corollary \ref{Ext13} applies and yields $u \in C^{1,\eta}\left( \overline{\Omega_u^+ \cap B_1 }\right)$ and that $ \Gamma^{\mathrm{int}}_{\mathrm{TP}}(u) \cap B_1$ is a $C^{1,\eta}$ graph 
with small norm. 

Since $1-\partial_n u$ satisfies the uniform elliptic equation \eqref{Ext18}, we are able to  apply the Harncak inequality for $1-\partial_n u \geq 0 $ in $\{x_n > \epsilon_0 \}$, and noting that $(1-\partial_n u)(\textbf{e}_n)=\overline{\eta}$, gives that
$$ 1-\partial_n u \geq c \overline{\eta}, \qquad \text{in} \quad B_{\frac{1}{4}} \left( \frac{\textbf{e}_n}{2}  \right) \subset \Omega_u^+ \cap B_1. $$
In $ \left(\Omega_u^+ \cap B_1 \right) \setminus B_{\frac{1}{4}} \left( \frac{\textbf{e}_n}{2}  \right)$, we compare $1-\partial_n u$ with a suitable barrier function $w$, which is a solution to \eqref{Ext18} with the following boundary values
\[\left\{\begin{array}{ll}
w= 0,  \quad & \text{on} \quad \partial \Omega_u^+ \cap B_1,\\
w= c \overline{\eta},  \quad & \text{on} \quad \partial B_{\frac{1}{4}} \left( \frac{\textbf{e}_n}{2}  \right),\\
w= 0, \quad & \text{on} \quad \partial B_1 \cap \Omega_u^+.
\end{array}\right.\] 
Using the Hopf lemma in $C^{1,\eta}$ domains, together with $C^{1,\eta}$ estimates up to the boundary, we conclude that $w$ grows linearly away from $\partial \Omega_u^+ \cap B_1$; hence
\begin{equation*}
\partial_n u(t\textbf{e}_n) \leq 1 - c \overline{\eta} t, \qquad \forall t \in \left[0,\frac{1}{4} \right],
\end{equation*}
with $c$ independent of $u$. 
Integrating along the $\textbf{e}_n$-direction, and using that $u(0)=0$ and $c,\overline{\eta}$ are universal, we find 
\begin{equation}
\label{Ext29}
u(t\textbf{e}_n) \leq t - c_1 t^2, \qquad \forall t \in \left[0,\frac{1}{4} \right],
\end{equation}
for some small universal constant $c_1>0$.

\medskip
{\bf Step 4. Iteration of the Nonlinear dichotomy.}
We now apply Proposition \ref{Ext07} and conclude that either alternative $(i)$ or $(ii)$ is satisfied. However, if $\epsilon''$ is small enough, then alternative $(i)$ cannot hold;  otherwise, by \eqref{Ext25}
$$ 1-c_0\epsilon_0 \geq |\nabla u^+(0)| \geq \partial_n u^+(0) \geq 1-\epsilon'', $$
which leads to a contradiction. 
Thus $(ii)$ must hold and we can rescale and iterate one more time; see Remark \ref{rmk-iterate}.
Similarly, after applying the conclusion of  Proposition \ref{Ext07} for $m$-times, we obtain that if $\epsilon''$ is sufficiently small depending on $m$, $\epsilon_0$, $c_0$ and  $r_0$, then only alternative $(ii)$ can hold for the $m$ iterations and conclude that
\begin{equation}
\label{ExtExt}
|u-P_{M,\nu}| \leq \epsilon_0 r^{2+\alpha_0}, \qquad \text{in} \quad B_r, \quad r=r_0^m,
\end{equation}
with
\begin{equation}
\label{Ext30}
\left| \frac{\nabla u^+(0)}{|\nabla u^+(0)|}-\nu \right| \leq C \epsilon_0 r^{1+\alpha_0}, \qquad \text{and} \qquad M\nu=0.
\end{equation}
Here we have also used Corollary \ref{Ext13}. 
Now, notice that
$$ 1-\epsilon'' \leq\partial_n u^+(0) =\textbf{e}_n \cdot \nabla u^+(0) \leq |\nabla u^+(0)| \leq 1, $$
which implies
$$ \left| \frac{\nabla u^+(0)}{|\nabla u^+(0)|}-\textbf{e}_n \right| \leq (2\epsilon'')^{\frac{1}{2}}. $$
Thus, if $\epsilon''$ is small enough depending on $r$, the inequality above together with \eqref{Ext30} gives that
$$ |\nu-\textbf{e}_n| \leq 2 C \epsilon_0 r^{1+\alpha_0}. $$
Since $M\nu=0$ and $\|M\|\leq 1$,
$$ \left| P_{M,\nu} \left( \frac{r}{2} \textbf{e}_n \right) - \frac{r}{2} \right| = \left| \frac{r}{2} \textbf{e}_n \cdot (\nu-\textbf{e}_n) + \frac{r^2}{8} (\textbf{e}_n-\nu)^TM(\textbf{e}_n-\nu) \right| \leq C' r^{2+\alpha_0}, $$
which combined with \eqref{ExtExt} gives that
$$ \left| u \left( \frac{r}{2} \textbf{e}_n \right) - \frac{r}{2} \right| \leq 2C' r^{2+\alpha_0}. $$
This contradicts \eqref{Ext29} for $t=\frac{r}{2}$, as long as $r$ is small enough universal, i.e. $m$ is large enough, and $\epsilon''$ is small enough.
This completes the proof of the theorem.
\end{proof}

The main consequence of Theorem \ref{Ext02} is the regularity of  the free boundary at regular points, namely Theorem \ref{T1}.

\begin{proof}[Proof of Theorem \normalfont{\ref{T1}}]
Consider a regular two-phase point of the free boundary. 
Then, the Lipschitz regularity of the solution, \cite[Theorem 1.2]{fotouhi-bayrami2023}, and the nondegeneracy property, \cite[Proposition 2.2]{fotouhi-bayrami2023}, proves that  any blow-ups at this point  will be a Lipschitz global solution with $0 \in \Gamma^{\mathrm{int}}_{\mathrm{TP}}(u)$. 
According to Theorem \ref{Ext02},  such a global solution  must be a two-plane solution, and we get that
 the (initial) flatness assumption in \cite[Theorem 1.1]{fotouhi-bayrami2023} is satisfied. 
 Therefore, by \cite[Theorem 1.1]{fotouhi-bayrami2023}, the boundaries $\partial \Omega_u^{\pm}$ are $C^{1,\eta}$ smooth in a neighborhood of the regular two-phase point, for any  $\eta \in (0,\frac{1}{3})$.
\end{proof}


\appendix
\section*{Appendix}
\setcounter{section}{1}
\setcounter{theorem}{0}

In the following lemma, for the convenience of the readers, we prove a Harnack inequality for an auxiliary problem of the $p$-Laplacian type.

\begin{lemma}[Harnack inequality]
\label{auxiliaryLemma}
Let $1<p<\infty$, and $\Omega$ be a smooth bounded domain in $\mathbb{R}^n$. Suppose $x_0 \in \Omega$, and $0<R \leq 1$ such that $\overline{B_{4R}(x_0)} \subset \Omega$. 
Let $h \in W^{1,p}(\Omega) \cap L^{\infty}(\Omega)$ be a non-negative solution to 
\begin{equation}
\label{AppE-1}
\mathrm{div} \left( \left| \nabla h(x) + Fx + E \right|^{p-2} \left( \nabla h(x) + Fx + E \right) \right)=0, \qquad \text{in} \quad \Omega,
\end{equation}
where $F$ is an $ n \times n$ symmetric matrix and $E \in \mathbb{R}^{n}$. 
Then, there exists $C$ depending only on $n$, $p$, $F$, $E$, and $\|h\|_{L^{\infty}(B_{4R}(x_0))}$ such that 
\begin{equation*}
\sup_{B_R(x_0)} h \leq C \left( \inf_{B_R(x_0)} h + R\right). 
\end{equation*}
\end{lemma}

\begin{proof}
Let $A: \Omega \times \mathbb{R}^n \to \mathbb{R}^n$ be defined as
$$ A(x, \xi)=\left| \xi + Fx+ E \right|^{p-2} \left( \xi+ Fx+ E \right). $$
Then, \eqref{AppE-1} can be written in the form
$$
\mathrm{div\,} A(x, \nabla u) = 0, \qquad \text{in} \quad \Omega.
$$
We verify the structural conditions needed for the Harnack inequality in \cite[Theorem 1.1]{MR3361842}.

First, for every $\xi \in \mathbb{R}^n$,
$$ \left| A(x, \xi) \right| = \left| \xi + Fx+ E \right|^{p-1} \leq C_1 \left( \left| \xi \right|^{p-1} + \max_{x \in \Omega} \left| Fx+E \right|^{p-1} \right) , $$
where $C_1$ depends only on $p$. 
Next, we compute
\begin{equation}
\label{AppE-3}
\begin{aligned}
\xi^T A(x,\xi) & = \left| \xi+ Fx+ E \right|^{p-2} \xi^T \left( \xi+ Fx+ E \right) \\
&= \left| \xi+ Fx+ E \right|^{p} - \left| \xi+ Fx+ E \right|^{p-2} \left(Fx+E \right)^T \left( \xi+ Fx+ E \right) \\
& \geq \left| \xi+ Fx+E \right|^{p} - \left( \max_{x \in \Omega} \left| Fx+E \right| \right) \left| \xi+ Fx+E \right|^{p-1}.
\end{aligned}
\end{equation}
Now, if $|\xi+ Fx+ E| \leq 2 \max_{x \in \Omega} \left| Fx+E \right| $, we get from \eqref{AppE-3}
\begin{equation}
\label{AppE-4}
\begin{aligned}
\xi^T A(x,\xi) & \geq \left| \xi+  Fx+ E \right|^{p} - 2^{p-1} \left(  \max_{x \in \Omega} \left| Fx+E \right|  \right)^p \\
& \geq C_2 \left| \xi \right|^{p} - C_2 \max_{x \in \Omega} \left| Fx+E \right|^{p} - 2^{p-1} \max_{x \in \Omega} \left| Fx+E \right|^{p},
\end{aligned}
\end{equation}
where $C_2$ depends only on $p$. 

On the other hand, if $|\xi+ Fx+ E| > 2 \max_{x \in \Omega} \left| Fx+E \right|$, we obtain from \eqref{AppE-3}
\begin{equation}
\label{AppE-5}
\begin{aligned}
\xi^T A(x,\xi) & \geq \left| \xi+ Fx+ E \right|^{p} - \left( \max_{x \in \Omega} \left| Fx+E \right| \right) \left| \xi+  Fx+E \right|^{p-1} \\
& = \left| \xi+ Fx+E \right|^{p} \left( 1- \left( \max_{x \in \Omega} \left| Fx+E \right| \right) |\xi+ Fx+ E|^{-1} \right) \\
& \geq \frac{1}{2} \left| \xi+ Fx+ E \right|^{p} \geq C_3 \left| \xi \right|^{p} - C_3 \max_{x \in \Omega}  \left| Fx+E \right|^{p},
\end{aligned}
\end{equation}
where $C_3$ depends only on $p$. Then, from \eqref{AppE-4} and \eqref{AppE-5}, we deduce
$$ \xi^T A(x,\xi) \geq C_4 |\xi|^p - C_5, $$
where $C_4$ depends only on $p$, and $C_5$ depends on $F$, $E$, and $p$. 

Therefore, the operator satisfies the standard structural assumptions of \cite[Theorem 1.1]{MR3361842}, and the desired Harnack inequality follows.
\end{proof}

In the following lemma, we present a Liouville type result for the global solutions of the operator $\mathcal{L}_p$, introduced in \eqref{Linearizedop}.

\begin{lemma}[Liouville-type result]
\label{C-G-S}
Let $U$ be a global solution to
\begin{equation}
\label{Di-18}
\begin{cases}
\mathcal{L}_p (U) = 0, & \qquad \text{in} \quad \mathbb{R}^n \cap \{x_n \neq 0\}, \\
\partial_n^+ U=g \partial_n^- U, & \qquad \text{on} \quad \{x_n =0\},
\end{cases}
\end{equation}
with $0<b_0 \leq g \leq b_1$ and
\begin{equation}
\label{Di-19}
\partial_n U \leq0, \qquad |U(x)|\leq C|x|^2,
\end{equation}
for some constant $C$. Then, $U=Q(x')$ where $Q(x')$ is a pure quadratic polynomial in the $x'$-direction such that $\mathcal{L}_p (Q) =\Delta Q = 0$.
\end{lemma}
\begin{proof}
   This follows directly from \cite[Lemma 4.7]{MR3998636}.
\end{proof}

\medskip

\paragraph{\bf{Acknowledgements}}
Masoud Bayrami and Morteza Fotouhi was supported by Iran National Science Foundation (INSF) under project No. 4031333.

\section*{Declarations}

\noindent {\bf  Data availability statement:} All data needed are contained in the manuscript.

\medskip
\noindent {\bf  Funding and/or Conflicts of interests/Competing interests:} The authors declare that there are no financial, competing or conflict of interests.

\addcontentsline{toc}{section}{\numberline{}References}

\bibliographystyle{acm}
\bibliography{mybibfile}

@article {MR618549,
    AUTHOR = {Alt, H. W. and Caffarelli, L. A.},
     TITLE = {Existence and regularity for a minimum problem with free
              boundary},
   JOURNAL = {J. Reine Angew. Math.},
  FJOURNAL = {Journal f\"{u}r die Reine und Angewandte Mathematik. [Crelle's
              Journal]},
    VOLUME = {325},
      YEAR = {1981},
     PAGES = {105--144},
      ISSN = {0075-4102},
   MRCLASS = {49A22 (35R35)},
  MRNUMBER = {618549},
MRREVIEWER = {Michel Chipot},
}

@article {MR682265,
    AUTHOR = {Alt, H. W. and Caffarelli, L. A. and Friedman, A.},
     TITLE = {Axially symmetric jet flows},
   JOURNAL = {Arch. Rational Mech. Anal.},
  FJOURNAL = {Archive for Rational Mechanics and Analysis},
    VOLUME = {81},
      YEAR = {1983},
    NUMBER = {2},
     PAGES = {97--149},
      ISSN = {0003-9527},
   MRCLASS = {35Q05 (76B10)},
  MRNUMBER = {682265},
MRREVIEWER = {Charles J. Amick},
       DOI = {10.1007/BF00250648},
       URL = {},
}

@article {MR752578,
    AUTHOR = {Alt, Hans Wilhelm and Caffarelli, Luis A. and Friedman, Avner},
     TITLE = {A free boundary problem for quasilinear elliptic equations},
   JOURNAL = {Ann. Scuola Norm. Sup. Pisa Cl. Sci. (4)},
  FJOURNAL = {Annali della Scuola Normale Superiore di Pisa. Classe di
              Scienze. Serie IV},
    VOLUME = {11},
      YEAR = {1984},
    NUMBER = {1},
     PAGES = {1--44},
      ISSN = {0391-173X},
   MRCLASS = {49A22 (35J60 35R35)},
  MRNUMBER = {752578},
MRREVIEWER = {Gioconda Moscariello},
       URL = {},
}

@article {MR733897,
    AUTHOR = {Alt, Hans Wilhelm and Caffarelli, Luis A. and Friedman, Avner},
     TITLE = {Jets with two fluids. {I}. {O}ne free boundary},
   JOURNAL = {Indiana Univ. Math. J.},
  FJOURNAL = {Indiana University Mathematics Journal},
    VOLUME = {33},
      YEAR = {1984},
    NUMBER = {2},
     PAGES = {213--247},
      ISSN = {0022-2518},
   MRCLASS = {35R35 (76B10 76T05)},
  MRNUMBER = {733897},
MRREVIEWER = {Peter Tolksdorf},
       DOI = {10.1512/iumj.1984.33.33011},
       URL = {},
}

@article {MR740956,
    AUTHOR = {Alt, Hans Wilhelm and Caffarelli, Luis A. and Friedman, Avner},
     TITLE = {Jets with two fluids. {II}. {T}wo free boundaries},
   JOURNAL = {Indiana Univ. Math. J.},
  FJOURNAL = {Indiana University Mathematics Journal},
    VOLUME = {33},
      YEAR = {1984},
    NUMBER = {3},
     PAGES = {367--391},
      ISSN = {0022-2518},
   MRCLASS = {35R35 (76B10 76T05)},
  MRNUMBER = {740956},
MRREVIEWER = {Peter Tolksdorf},
       DOI = {10.1512/iumj.1984.33.33021},
       URL = {},
}

@article {MR772122,
    AUTHOR = {Alt, Hans Wilhelm and Caffarelli, Luis A. and Friedman, Avner},
     TITLE = {Compressible flows of jets and cavities},
   JOURNAL = {J. Differential Equations},
  FJOURNAL = {Journal of Differential Equations},
    VOLUME = {56},
      YEAR = {1985},
    NUMBER = {1},
     PAGES = {82--141},
      ISSN = {0022-0396},
   MRCLASS = {35J20 (35Q99 49A22 76N10)},
  MRNUMBER = {772122},
MRREVIEWER = {Erich Miersemann},
       DOI = {10.1016/0022-0396(85)90101-9},
       URL = {},
}

@article{fotouhi-bayrami2023,
  title={Regularity in the two-phase Bernoulli problem for the p-Laplace operator},
  author={Bayrami, Masoud and Fotouhi, Morteza},
  journal={Calculus of Variations and Partial Differential Equations},
  volume={63},
  number={7},
  pages={183},
  year={2024},
  publisher={Springer}
}

@book {MR1351007,
    AUTHOR = {Caffarelli, Luis A. and Cabr\'{e}, Xavier},
     TITLE = {Fully nonlinear elliptic equations},
    SERIES = {American Mathematical Society Colloquium Publications},
    VOLUME = {43},
 PUBLISHER = {American Mathematical Society, Providence, RI},
      YEAR = {1995},
     PAGES = {vi+104},
      ISBN = {0-8218-0437-5},
   MRCLASS = {35J60 (35-01 35B45 35B65 35Dxx)},
  MRNUMBER = {1351007},
MRREVIEWER = {P. Lindqvist},
       DOI = {10.1090/coll/043},
       URL = {},
}

@article {MR2133664,
    AUTHOR = {Danielli, Donatella and Petrosyan, Arshak},
     TITLE = {A minimum problem with free boundary for a degenerate
              quasilinear operator},
   JOURNAL = {Calc. Var. Partial Differential Equations},
  FJOURNAL = {Calculus of Variations and Partial Differential Equations},
    VOLUME = {23},
      YEAR = {2005},
    NUMBER = {1},
     PAGES = {97--124},
      ISSN = {0944-2669},
   MRCLASS = {35R35 (35B65 35J20 35J60 49N60)},
  MRNUMBER = {2133664},
MRREVIEWER = {Jes\'{u}s Hern\'{a}ndez},
       DOI = {10.1007/s00526-004-0294-5},
       URL = {},
}

@article {MR4285137,
    AUTHOR = {De Philippis, Guido and Spolaor, Luca and Velichkov, Bozhidar},
     TITLE = {Regularity of the free boundary for the two-phase Bernoulli
              problem},
   JOURNAL = {Invent. Math.},
  FJOURNAL = {Inventiones Mathematicae},
    VOLUME = {225},
      YEAR = {2021},
    NUMBER = {2},
     PAGES = {347--394},
      ISSN = {0020-9910},
   MRCLASS = {35R35 (47A75 49N60 49Q10 49R05)},
  MRNUMBER = {4285137},
       DOI = {10.1007/s00222-021-01031-7},
       URL = {},
}

@article {MR3998636,
    AUTHOR = {De Silva, Daniela and Savin, Ovidiu},
     TITLE = {Global solutions to nonlinear two-phase free boundary
              problems},
   JOURNAL = {Comm. Pure Appl. Math.},
  FJOURNAL = {Communications on Pure and Applied Mathematics},
    VOLUME = {72},
      YEAR = {2019},
    NUMBER = {10},
     PAGES = {2031--2062},
      ISSN = {0010-3640},
   MRCLASS = {35R35 (35J60)},
  MRNUMBER = {3998636},
MRREVIEWER = {Mariana Smit Vega Garcia},
       DOI = {10.1002/cpa.21811},
       URL = {},
}

@article {MR4273843,
    AUTHOR = {Ferrari, Fausto and Lederman, Claudia},
     TITLE = {Regularity of flat free boundaries for a {$p(x)$}-{L}aplacian
              problem with right hand side},
   JOURNAL = {Nonlinear Anal.},
  FJOURNAL = {Nonlinear Analysis. Theory, Methods \& Applications. An
              International Multidisciplinary Journal},
    VOLUME = {212},
      YEAR = {2021},
     PAGES = {Paper No. 112444, 25},
      ISSN = {0362-546X},
   MRCLASS = {35R35 (35B65 35J60 35J70)},
  MRNUMBER = {4273843},
       DOI = {10.1016/j.na.2021.112444},
       URL = {},
}

@article{ferrari2022regularity,
  title={Regularity of Lipschitz free boundaries for a $p(x)$-Laplacian problem with right hand side},
  author={Ferrari, Fausto and Lederman, Claudia},
  journal={Journal de Math{\'e}matiques Pures et Appliqu{\'e}es},
  year={2022},
  publisher={Elsevier}
}

@article{karakhanyan2021regularity,
  title={Regularity for the two-phase singular perturbation problems},
  author={Karakhanyan, Aram},
  journal={Proceedings of the London Mathematical Society},
  volume={123},
  number={5},
  pages={433--459},
  year={2021},
  publisher={Wiley Online Library}
}

@book{velichkov2019regularity,
  author       = {Velichkov, Bozhidar},
  title        = {Regularity of the One-phase Free Boundaries},
  series       = {Lecture Notes of the Unione Matematica Italiana},
  volume       = {28},
  publisher    = {Springer},
  address      = {Cham},
  year         = {2023},
  doi          = {10.1007/978-3-031-13238-4},
  isbn         = {978-3-031-13237-7},
  pages        = {xiii, 247}
}

@article {MR3361842,
    AUTHOR = {Wolanski, Noemi},
     TITLE = {Local bounds, {H}arnack's inequality and {H}\"{o}lder continuity
              for divergence type elliptic equations with non-standard
              growth},
   JOURNAL = {Rev. Un. Mat. Argentina},
  FJOURNAL = {Revista de la Uni\'{o}n Matem\'{a}tica Argentina},
    VOLUME = {56},
      YEAR = {2015},
    NUMBER = {1},
     PAGES = {73--105},
      ISSN = {0041-6932},
   MRCLASS = {35J62 (35B45 35B65 35J70 35J92)},
  MRNUMBER = {3361842},
}

@article {MR732100,
    AUTHOR = {Alt, Hans Wilhelm and Caffarelli, Luis A. and Friedman, Avner},
     TITLE = {Variational problems with two phases and their free
              boundaries},
   JOURNAL = {Trans. Amer. Math. Soc.},
  FJOURNAL = {Transactions of the American Mathematical Society},
    VOLUME = {282},
      YEAR = {1984},
    NUMBER = {2},
     PAGES = {431--461},
      ISSN = {0002-9947},
   MRCLASS = {49A29 (35J85)},
  MRNUMBER = {732100},
MRREVIEWER = {Peter Tolksdorf},
       DOI = {10.2307/1999245},
       URL = {},
}

@article {MR990856,
    AUTHOR = {Caffarelli, Luis A.},
     TITLE = {A {H}arnack inequality approach to the regularity of free
              boundaries. {I}. {L}ipschitz free boundaries are
              {$C^{1,\alpha}$}},
   JOURNAL = {Rev. Mat. Iberoamericana},
  FJOURNAL = {Revista Matem\'{a}tica Iberoamericana},
    VOLUME = {3},
      YEAR = {1987},
    NUMBER = {2},
     PAGES = {139--162},
      ISSN = {0213-2230},
   MRCLASS = {35R35 (35J05)},
  MRNUMBER = {990856},
MRREVIEWER = {Rolando Magnanini},
       DOI = {10.4171/RMI/47},
       URL = {},
}

@article {MR973745,
    AUTHOR = {Caffarelli, Luis A.},
     TITLE = {A {H}arnack inequality approach to the regularity of free
              boundaries. {II}. {F}lat free boundaries are {L}ipschitz},
   JOURNAL = {Comm. Pure Appl. Math.},
  FJOURNAL = {Communications on Pure and Applied Mathematics},
    VOLUME = {42},
      YEAR = {1989},
    NUMBER = {1},
     PAGES = {55--78},
      ISSN = {0010-3640},
   MRCLASS = {35R35 (35B65 35J25 49A21)},
  MRNUMBER = {973745},
MRREVIEWER = {Erich Miersemann},
       DOI = {10.1002/cpa.3160420105},
       URL = {},
}

@article {MR1029856,
    AUTHOR = {Caffarelli, Luis A.},
     TITLE = {A {H}arnack inequality approach to the regularity of free
              boundaries. {III}. {E}xistence theory, compactness, and
              dependence on {$X$}},
   JOURNAL = {Ann. Scuola Norm. Sup. Pisa Cl. Sci. (4)},
  FJOURNAL = {Annali della Scuola Normale Superiore di Pisa. Classe di
              Scienze. Serie IV},
    VOLUME = {15},
      YEAR = {1988},
    NUMBER = {4},
     PAGES = {583--602 (1989)},
      ISSN = {0391-173X},
   MRCLASS = {35R35 (35B65 35J25 49J10)},
  MRNUMBER = {1029856},
MRREVIEWER = {Erich Miersemann},
       URL = {},
}

@book {MR2145284,
    AUTHOR = {Caffarelli, Luis A. and Salsa, Sandro},
     TITLE = {A geometric approach to free boundary problems},
    SERIES = {Graduate Studies in Mathematics},
    VOLUME = {68},
 PUBLISHER = {American Mathematical Society, Providence, RI},
      YEAR = {2005},
     PAGES = {x+270},
      ISBN = {0-8218-3784-2},
   MRCLASS = {35R35 (35-01 35J25 35K20)},
  MRNUMBER = {2145284},
MRREVIEWER = {Luca Lorenzi},
       DOI = {10.1090/gsm/068},
       URL = {},
}

@article {MR2680176,
    AUTHOR = {Lewis, John L. and Nystr\"{o}m, Kaj},
     TITLE = {Regularity of {L}ipschitz free boundaries in two-phase
              problems for the {$p$}-{L}aplace operator},
   JOURNAL = {Adv. Math.},
  FJOURNAL = {Advances in Mathematics},
    VOLUME = {225},
      YEAR = {2010},
    NUMBER = {5},
     PAGES = {2565--2597},
      ISSN = {0001-8708},
   MRCLASS = {35R35 (35B65 35J70 35J92)},
  MRNUMBER = {2680176},
MRREVIEWER = {Pierre Dreyfuss},
       DOI = {10.1016/j.aim.2010.05.005},
       URL = {},
}

@article {MR2876248,
    AUTHOR = {Lewis, John L. and Nystr\"{o}m, Kaj},
     TITLE = {Regularity of flat free boundaries in two-phase problems for
              the {$p$}-{L}aplace operator},
   JOURNAL = {Ann. Inst. H. Poincar\'{e} C Anal. Non Lin\'{e}aire},
  FJOURNAL = {Annales de l'Institut Henri Poincar\'{e} C. Analyse Non Lin\'{e}aire},
    VOLUME = {29},
      YEAR = {2012},
    NUMBER = {1},
     PAGES = {83--108},
      ISSN = {0294-1449},
   MRCLASS = {35R30 (35B65 35J70 35J92)},
  MRNUMBER = {2876248},
MRREVIEWER = {Vy Khoi Le},
       DOI = {10.1016/j.anihpc.2011.09.002},
       URL = {},
}

@article{karakhanyan2021full,
  title={Full and partial regularity for a class of nonlinear free boundary problems},
  author={Karakhanyan, Aram},
  journal={Annales de l'Institut Henri Poincar{\'e} C},
  volume={38},
  number={4},
  pages={981--999},
  year={2021}
}

@article {MR2250499,
    AUTHOR = {Danielli, Donatella and Petrosyan, Arshak},
     TITLE = {Full regularity of the free boundary in a {B}ernoulli-type
              problem in two dimensions},
   JOURNAL = {Math. Res. Lett.},
  FJOURNAL = {Mathematical Research Letters},
    VOLUME = {13},
      YEAR = {2006},
    NUMBER = {4},
     PAGES = {667--681},
      ISSN = {1073-2780},
   MRCLASS = {35R35 (35J20 35J65 49N60)},
  MRNUMBER = {2250499},
MRREVIEWER = {Eduardo V. Teixeira},
       DOI = {10.4310/MRL.2006.v13.n4.a14},
       URL = {},
}

@article {MR2399040,
    AUTHOR = {Petrosyan, Arshak},
     TITLE = {On the full regularity of the free boundary in a class of
              variational problems},
   JOURNAL = {Proc. Amer. Math. Soc.},
  FJOURNAL = {Proceedings of the American Mathematical Society},
    VOLUME = {136},
      YEAR = {2008},
    NUMBER = {8},
     PAGES = {2763--2769},
      ISSN = {0002-9939},
   MRCLASS = {35R35 (35B65 35J20 35J92 49J10)},
  MRNUMBER = {2399040},
MRREVIEWER = {Ken Shirakawa},
       DOI = {10.1090/S0002-9939-08-09476-8},
       URL = {},
}

@article {MR3771123,
    AUTHOR = {Dipierro, Serena and Karakhanyan, Aram L.},
     TITLE = {Stratification of free boundary points for a two-phase
              variational problem},
   JOURNAL = {Adv. Math.},
  FJOURNAL = {Advances in Mathematics},
    VOLUME = {328},
      YEAR = {2018},
     PAGES = {40--81},
      ISSN = {0001-8708},
   MRCLASS = {35R35 (35J20 35J92)},
  MRNUMBER = {3771123},
MRREVIEWER = {Carlos V\'{a}zquez Cend\'{o}n},
       DOI = {10.1016/j.aim.2018.01.005},
       URL = {},
}

@book {MR1814364,
    AUTHOR = {Gilbarg, David and Trudinger, Neil S.},
     TITLE = {Elliptic partial differential equations of second order},
    SERIES = {Classics in Mathematics},
      NOTE = {Reprint of the 1998 edition},
 PUBLISHER = {Springer-Verlag, Berlin},
      YEAR = {2001},
     PAGES = {xiv+517},
      ISBN = {3-540-41160-7},
   MRCLASS = {35-02 (35Jxx)},
  MRNUMBER = {1814364},
}

@article {MR3218810,
    AUTHOR = {De Silva, Daniela and Ferrari, Fausto and Salsa, Sandro},
     TITLE = {Two-phase problems with distributed sources: regularity of the
              free boundary},
   JOURNAL = {Anal. PDE},
  FJOURNAL = {Analysis \& PDE},
    VOLUME = {7},
      YEAR = {2014},
    NUMBER = {2},
     PAGES = {267--310},
      ISSN = {2157-5045},
   MRCLASS = {35R35 (35A09 35B65 35D40 35J05)},
  MRNUMBER = {3218810},
MRREVIEWER = {Erich Miersemann},
       DOI = {10.2140/apde.2014.7.267},
       URL = {},
}

@article {MR2813524,
    AUTHOR = {De Silva, D.},
     TITLE = {Free boundary regularity for a problem with right hand side},
   JOURNAL = {Interfaces Free Bound.},
  FJOURNAL = {Interfaces and Free Boundaries. Mathematical Analysis,
              Computation and Applications},
    VOLUME = {13},
      YEAR = {2011},
    NUMBER = {2},
     PAGES = {223--238},
      ISSN = {1463-9963},
   MRCLASS = {35R35 (35B45 35B65 35D40 35J15)},
  MRNUMBER = {2813524},
MRREVIEWER = {Alain Brillard},
       DOI = {10.4171/IFB/255},
       URL = {},
}

@article{savin2023regularity,
  author    = {Savin, Ovidiu and Yu, Hui},
  title     = {Regularity of the singular set in the fully nonlinear obstacle problem},
  journal   = {Journal of the European Mathematical Society (JEMS)},
  volume    = {25},
  number    = {2},
  pages     = {571--610},
  year      = {2023},
  doi       = {10.4171/JEMS/1190}
}

@article{fotouhi2024minimization,
  title={A minimization problem with free boundary for p-Laplacian weakly coupled system},
  author={Fotouhi, Morteza and Shahgholian, Henrik},
  journal={Advances in Nonlinear Analysis},
  volume={13},
  number={1},
  pages={20230138},
  year={2024},
  publisher={De Gruyter}
}

@article{bayrami2024lipschitz,
  title={Lipschitz regularity of a weakly coupled vectorial almost-minimizers for the p-Laplacian},
  author={Bayrami, Masoud and Fotouhi, Morteza and Shahgholian, Henrik},
  journal={Journal of Differential Equations},
  volume={412},
  pages={447--473},
  year={2024},
  publisher={Elsevier}
}

\end{document}